\documentclass[12pt, reqno]{amsart}
\usepackage{amsmath}
\usepackage{amssymb,amsthm,amsmath}
\usepackage[numbers,sort&compress]{natbib}
\usepackage{amssymb,amsmath}
\usepackage{amsfonts}
\usepackage{mathrsfs}
\usepackage{latexsym}
\usepackage{amssymb}
\usepackage{amsthm}
\usepackage{hyperref}
\usepackage{booktabs}
\usepackage{multirow}
\usepackage{indentfirst}
\usepackage{xcolor}

\date{\today}

\let\oldsection\section
\renewcommand\section{\setcounter{equation}{0}\oldsection}

\newtheorem{corollary}{Corollary}[section]
\newtheorem{theorem}{Theorem}[section]
\newtheorem{lemma}{Lemma}[section]
\newtheorem{proposition}{Proposition}[section]

\newtheorem{remark}{Remark}[section]

\begin{document}
\title[Global well-posedness of inviscid resistive isentropic MHD]{Global well-posedness of the inviscid resistive isentropic compressible MHD system}
\author{Jinkai Li}
\address{Jinkai Li, South China Research Center for Applied Mathematics and Interdisciplinary Studies, School of Mathematical Sciences, South China Normal University, Zhong Shan Avenue West 55, Guangzhou 510631, P. R. China}
\email{jklimath@m.scnu.edu.cn; jklimath@gmail.com}
\author{Liening~Qiao$^\ast$}
\address[Liening~Qiao]{School of Mathematical Sciences, South China Normal University, Zhong Shan Avenue West 55, Guangzhou 510631, P. R. China}
\email{l\_n\_qiao@163.com}
\allowdisplaybreaks
\maketitle
\footnotetext{$^\ast$Corresponding author.}

\allowdisplaybreaks

\begin{abstract}
Due to the absence of dissipation mechanism to the inviscid compressible systems, it is a challenging problem to
prove their global solvability.
In this paper, we are concerned with the initial-boundary value problem to the inviscid and resistive isentropic compressible
magnetohydrodynamic (MHD) system on three dimensional torus $\mathbb T^3$. Global well-posedness and large time behavior of solutions are
established in the first time for the isentropic setting, under the condition that the initial data $(\rho_0, u_0, H_0)$ is a small
perturbation around the constant state $(1, 0, w)$, with $w$ satisfying the Diophantine condition.
The main observation of this paper is that the spatial derivatives of the density along
directions perpendicular to $w$ are dissipated. Such dissipation mechanism
is generated from the interaction between the velocity field and the
background magnetic field. This verifies the weak stabilizing effects of the magnetic filed on the dynamics in the scenario
of inviscid isentropic flows. Due
to different dissipation mechanisms for the density, velocity, and magnetic field, three ties of dissipative energies are designed, that is,
high order Sobolev norms of the perturbed magnetic field, intermediate order Sobolev norms of the perturbed
density, and low order Sobolev norms of the velocity field.
\end{abstract}



\section{Introduction and Main Results}
\label{sec1}
The magnetohydrodynamic (MHD) model is a system describing the dynamics of conducting fluid under the effect of the magnetic field and finds its way in a huge range of physical
 objects. Mathematically, the compressible MHD system consists of the compressible Navier-Stokes equations of fluid dynamics and the Maxwell equations of electromagnetism (see, e.g., \cite{Cabannes-H,Landau-Lifshitz,Polovin-Demutskii}), that is
\begin{equation}\label{MHD-begin}
\left\{
\begin{array}{rl}
&\hspace{-0,5cm}\partial_t\rho+\textrm{div}(\rho u)=0,
\smallskip\\
&\hspace{-0,5cm}\rho(\partial_t u
+(u\cdot \nabla)u)
-\mu\Delta u-(\mu+\lambda)\nabla\textrm{div}u+\nabla P=(H\cdot\nabla)H,\quad P:=p+\frac{|H|^2}{2},
\medskip\\
&\hspace{-0,5cm}\partial_t H+(u\cdot \nabla)H-\nu\Delta H
=(H\cdot\nabla)u-H\textrm{div}u,
\medskip\\
&\hspace{-0,5cm}{\mathop{\rm div}}\, H=0,
\end{array}
\right.
\end{equation}
where $\rho\in\mathbb R^+$,  $u\in\mathbb R^3$, $H\in\mathbb R^3$ and $p=p(\rho)$ describe the density, velocity field, magnetic field, and
pressure, respectively, while $\mu$ and $\lambda$ are constant viscous coefficients, and $\nu$ is the the constant resistivity coefficient.

Due to its wide applications in physics and mathematical importance, the study of the MHD equations has been attracting many mathematicians
over the past decades and a lot of important results have been achieved.
In the case that with both viscosity and resistivity, that is $\mu>0$, $2\mu+\lambda>0,$ and $\nu>0,$ mathematical results for system
\eqref{MHD-begin} are similar to those of the compressible Navier-Stoles equations, see, e.g., \cite{Abidi-Paicu,CGZ,Xu-Qiao-Fu} for the
incompressible fluids and \cite{KawashimaS,SuenA,Gao-Wu-Xu} for the compressible fluids.

In the case that with only viscosity but without resistivity, that is $\mu>0$ and
$2\mu+\lambda>0$, but $\nu=0$, system \eqref{MHD-begin} becomes
the viscous non-resistive compressible MHD equations. In this case, the first result concerning global well-posedness of strong solutions was
established by Lin--Zhang \cite{Lin-Zhang} for the incompressible fluid, under the conditions that the fluid has low initial energy and that
the initial magnetic field is a small perturbation of the background $(0,0,1)$ (or $(0,1)$ in 2D case). For the spatially periodic initial
boundary value problem, global well-posedenss of low energy strong solutions was also established when the background magnetic field $w$
satisfies the so called Diophantine condition (see \eqref{Diophantine0}, in the below). The first result along this direction was obtained by
Chen--Zhang--Zhou \cite{Chen-Zhang-Zhou}, see Zhai \cite{ZhaiX}, Xie--Jiu--Liu \cite{Xie-Jiu-Liu} and Jiang--Jiang \cite{Jiang-Jang1} for
further developments. Global well-posedness results were also established for the compressible fluids, see \cite{Wu-Wu,Dong-Wu-Zhai} for the
Cauchy problems, \cite{Ren-Xiang-Zhang,Tan-Wang,ZhaoY} for the problem on infinite strip, and \cite{Jiu-Liu-Xie,Wu-Zhai,Wu-Zhu,Li-Xu-Zhai} for
the problem on torus. In particular, for the spatially periodic initial boundary value problem considered in
\cite{Jiu-Liu-Xie,Wu-Zhai,Li-Xu-Zhai}, the background magnetic field $w$ was assumed to satisfy the Diophantine condition. Recently,
Huang-Xin-Yan \cite{Huang-Xin-Yan} investigated the finite-time blow-up of solutions to the viscous non-resistive compressible MHD
equations with vacuum states. It was demonstrated that when the initial vacuum region contains a ball, the radial symmetric strong solution to
the initial-boundary value problem blows up in finite time.

In contrast to the cases stated in the above two paragraphes, not many results are currently available for the inviscid and resistive MHD
system \eqref{MHD}, i.e., $\mu=\lambda=0$ and $\nu>0$. In this case, most results considered the incompressible fluid.  For the homogeneous
incompressible fluid, global existence of weak solutions was established in \cite{Cao-Wu,Lei-Zhou}, see also
\cite{Hassainia,Jiu-Niu-Wu-Yu,Cao-Wu-Yuan,Jiu-Zhao,QiaoY,Ye-Yin} for some related results. Regarding the spatially periodic initial boundary
value problem, global well-posedness of strong solutions was established in \cite{Wei-Zhang,Chen-Zhang-Zhou,Xie-Jiu-Liu,ZhaiX,Zhou-Zhu}, where
\cite{Wei-Zhang} deals with the case with zero magnetic background, \cite{Chen-Zhang-Zhou,Xie-Jiu-Liu,ZhaiX,Zhou-Zhu} consider nonzero
magnetic background, while \cite{Zhou-Zhu} assumes some symmetry condition. For the plasma-vacuum interface problem, global well-posedness was
established by Wang--Xin \cite{Wang-Xin1}. Compared with the incompressible case, there are very few results on the global well-posedness
for the compressible system. Local well-posedness of the interface problem was proved by Zhang \cite{ZhangJ}. Concerning the global
well-posedness, we are aware of the results by Wang--Xin \cite{Wang-Xin} and Wu--Xu--Zhai \cite{Wu-Xu-Zhai}.
For a small perturbed initial
magnetic field around a non-horizontal constant background magnetic field, Wang--Xin \cite{Wang-Xin} established global well-posedness for
the heat conducting flow in a three dimensional strip. Very recently, still for the heat conducting flow,
Wu--Xu--Zhai \cite{Wu-Xu-Zhai} considered global well-posedness in the case on $\mathbb T^3$,
where the background constant magnetic field is assumed to satisfy the Diophantine condition.
The arguments in \cite{Wang-Xin} and \cite{Wu-Xu-Zhai} crucially depend on the heat conducting dissipation
mechanism and thus do not apply to the non-heat conducting flow or the isentropic flow.

In this paper, we are concerned with the following inviscid and resistive isentropic MHD system  in 3D periodic domain $\mathbb T^3:=[0,1]^3$,
\begin{equation}\label{MHD}
\left\{
\begin{array}{rl}
&\hspace{-0,5cm}\partial_t\rho+\textrm{div}(\rho u)=0,\ (x,t)\in\mathbb T^3\times\mathbb R,
\smallskip\\
&\hspace{-0,5cm}\rho(\partial_t u
+(u\cdot \nabla) u)
+\nabla P=(H\cdot \nabla) H,\quad P:=p+\frac{|H|^2}{2},
\medskip\\
&\hspace{-0,5cm}\partial_t H+(u\cdot \nabla) H-\nu\Delta H
=(H\cdot \nabla) u-H\textrm{div}u,
\medskip\\
&\hspace{-0,5cm}{\mathop{\rm div}}\, H=0,
\medskip\\
&\hspace{-0,5cm}(\rho,u,H)_{|t=0}=(\rho_0,u_0,H_0),
\end{array}
\right.
\end{equation}
where $p=p(\rho)$ with $p(\cdot)$ being a smooth function satisfying $p'>0$ in $(0,\infty)$.
For simplicity, we set $\nu=1$,
as its exact value does not affect the mathematical analysis.

Recall that in the heat conducting case, i.e., for the inviscid and resistive heat conducting MHD system, the corresponding global
well-posedness of solution was recently established in \cite{Wang-Xin,Wu-Xu-Zhai}, where the main observation there is that the heat conduction produces
the dissipation of the divergence and further the full dissipation of the velocity. Unfortunately, such dissipation mechanism generated by the
heat conduction is not expectable to the isentropic flow due to the absence of the
temperature equation in the isentropic system (\ref{MHD}). Therefore, some new dissipation mechanism is required. To our best
knowledge, global well-posedness result to system (\ref{MHD}) is not available in the existing literature. The aim of this paper is address the global well-posedness of solutions around nonzero constant magnetic background.

Given a nonzero constant vector $w\in\mathbb R^3$. We call that $w$ satisfies the Diophantine condition, if there are two positive constants $c$ and $r>2$ such that
\begin{equation}
\label{Diophantine0}
|w\cdot k|\geq \frac{c}{|k|^r},\quad \forall k\in\mathbb{Z}^3\backslash{\{0\}}.
\end{equation}
As shown in \cite{Chen-Zhang-Zhou}, the Diophantine condition is satisfied for almost
all $w\in\mathbb R^3.$ The main advantage of Diophantine condition is that it provides a Poincar\'{e} type inequality with the price of
derivative loss,
see Lemma \ref{Diophantine-inequality}, Lemma \ref{Diophantine-inequality0}, and Corollary \ref{CorPoincare}, in the below. Set
$$
h:=H-w\quad\text{and}\quad h_0:=H_0-w.
$$
Then, \eqref{MHD} can be rewritten as
\begin{equation}\label{MHD1}
\left\{
\begin{array}{rl}
&\hspace{-0,5cm}\partial_t\rho+\textrm{div}(\rho u)=0,
\smallskip\\
&\hspace{-0,5cm}\rho(\partial_t u
+(u\cdot \nabla) u)
+\nabla p=(w\cdot \nabla) h+(h\cdot \nabla) h-\nabla(w\cdot h)-\frac12\nabla(|h|^2),
\medskip\\
&\hspace{-0,5cm}\partial_t h+(u\cdot \nabla) h-\Delta h
=(w\cdot \nabla) u+(h\cdot \nabla) u-w\textrm{div}u-h\textrm{div}u,
\medskip\\
&\hspace{-0,5cm}{\mathop{\rm div}}\, h=0,
\medskip\\
&\hspace{-0,5cm}(\rho,u,h)_{|t=0}=(\rho_0,u_0,h_0).
\end{array}
\right.
\end{equation}

Set $\Lambda=\sqrt{-\Delta}$ and denote by $L^q(\mathbb T^3)$ and $H^k(\mathbb T^3)$ with $q\geq1$ and $k\geq0$ the usual Lebesgue and Sobolev
spaces on $\mathbb T^3$ with norms $\|\cdot\|_{L^q}$ and $\|\cdot\|_k$, and set $H^0(\mathbb T^3)=L^2(\mathbb T^3)$. For simplicity, we use
the notation $``f\lesssim g"$ to represent $``f\leq Cg"$ for a general
constant $C$. For simplicity, we adopt
\begin{eqnarray*}
&\int fdx=\int_{\mathbb T^3}fdx,\quad\|f,g\|^2_{X}=\|f\|^2_{X}+\|g\|^2_{X},\quad\|h(f,g)\|^2_{X}=\|hf,hg\|^2_{X},
\end{eqnarray*}
and commutator
$$
[U,V]W:=U(VW)-V(UW).
$$

Set
\begin{equation*}\label{alpha}
\alpha_1:=\min\left\{\inf_{s\in(\frac12,\frac32)}\frac{p'(s)}{s},\frac12\right\},\ \ \alpha_2:=\max\left\{\sup_{s\in(\frac12,\frac32)}\frac{p'(s)}{s},\frac32\right\},
\end{equation*}
and
$$
Q:=\sup_{0\leq k\leq
N+d+3,s\in(\frac12,\frac32)}|p^{(k)}(s)|,\quad\beta:=p'(1).
$$
For any positive integer $j$, denote
\begin{eqnarray}
\label{Ej}
  \mathcal E_j(t):=\|a, u, h\|_j^2(t),\quad\mathcal E_{j,0}:=\mathcal E_{j}(0), \label{EN0}
\end{eqnarray}
where
\begin{equation}
  \label{a}
  a:=\rho-1.
\end{equation}
Assume that
\begin{equation}\label{junzhi0}
\begin{split}
\int \rho_0dx=1,\quad \int \rho_0u_0dx=\int h_0dx=0,
\end{split}
\end{equation}
then it follows
\begin{equation}\label{junzhi1}
\begin{split}
\int \rho dx=1,\quad \int\rho u dx=\int h dx=0.
\end{split}
\end{equation}

We are now in the position to state the main result of this paper.

\begin{theorem}\label{th:main1}
Given arbitrary constant vector $w\in\mathbb R^3$ satisfying the Diophantine condition (\ref{Diophantine0})
for some positive constants $c$ and $r$. Let $L, M, N, d$ be any given positive integers such that
\begin{equation*}
  r+3\leq L\leq M-r-1,\quad M\leq N-r-2,\quad\mbox{ and }\quad d>2(N+r-L).
\end{equation*}
Assume that the initial datum $(\rho_0,u_{0},h_0)\in H^{N+d}$ satisfies (\ref{junzhi0}) and $\inf_{x\in\mathbb T^3}\rho_0(x)$.

Then, there is a positive constant $\varepsilon_0$ depending only on $N$, $d$, $L$, $|w|$, $r$, $c$, $\beta,$ $\alpha_1$, $\alpha_2,$ and $Q$,
such that the following two
items hold:

(i) System \eqref{MHD1}
admits a unique global solution $(\rho, u, h)\in C(\mathbb[0,\infty);H^{N+d})$,
as long as
\begin{equation*}
  \mathcal E_{N,0}\left(1+\mathcal E_{N+d,0}^{\frac{2(N+r-L)}{d-2(N+r-L)}}\right)\leq\varepsilon_0,
\end{equation*}
where $\mathcal E_{N,0}$ and $\mathcal E_{N+d,0}$ are given as in (\ref{Ej}).

(ii) Moreover, the unique global solution obtained in (i) satisfies
 \begin{eqnarray*}
 && \frac34\leq\rho\leq\frac45\quad\mbox{on }\mathbb T^3\times[0,\infty), \quad \mathcal E_{N+d}(t)\leq C_0\mathcal E_{N+d,0},\\
&&\mathcal E_{N}(t)\leq C_0\mathcal E_{N,0}\left[1+
  \left(\frac{\mathcal E_{N,0}}{ \mathcal E_{N+d,0}}\right)^{\frac{N+r-L}{d}} t\right]^{-\frac{d}{N+r-L}},
\end{eqnarray*}
for $t\in[0,\infty)$, where $C_0$ is a positive constant depending on depending only on $N$, $d$, $L$, $|w|$, $r$, $c$, $\beta,$ $\alpha_1$, $\alpha_2,$ and $Q$, and $\mathcal E_N(t)$ and $\mathcal E_{N+d}(t)$ are defined as
in (\ref{Ej}).
\end{theorem}

\begin{remark}\label{remar1}
To our best knowledge, this is the first global well-posedness result for the inviscid and resistive isentropic MHD system (\ref{MHD}).
If $H_0\equiv0$, then (\ref{MHD}) reduces to the compressible Euler system for which singularities will form in finite time,
see, e.g., \cite{BSV1,BSV2,BSV3,Cao-Serrano-Shi-Staffilani,CHRIST1,CHRIST2,LUKSPECK,MRRS1,SIDERIS,YIN}. So our result demonstrates that
the nonzero constant background magnetic field has stabilizing effect on the motion of the inviscid fluids.
\end{remark}

\begin{remark}\label{remar2}
Regarding the global well-posedness of the multi-dimensional inviscid and resistive compressible MHD system,
the available results in the existing literature are only for
the heat-conducting flow, see \cite{Wang-Xin,Wu-Xu-Zhai},
where the heat conducting plays a key role to control the divergence of the velocity.
The current paper gives the first result on the global well-posedness to the inviscid and resistive compressible MHD system
in the isentropic scenario.
\end{remark}

Key strategies of the proof are outlined as follows. The local-in-time well-posedness of system \eqref{MHD1} follows from
standard techniques. By continuity arguments, the main issue is to get suitable
time independent a priori estimates. The main difficulty arises from the absence of explicit dissipation or damping terms for the
density and velocity. To overcome this difficulty, we leverage
structural interactions with the background magnetic field to exploit the hidden dissipation mechanisms
for the density and velocity. Once these dissipative structures are identified, we integrate them with the inherent dissipation for the
magnetic field. This synthesis yields decaying estimates for $(a,u,h)$ (see Lemma \ref{lemma-pri4} for details),
ensuring the time integrability of $\|a,u,h\|_N$ and further the full energy estimates (see Proposition \ref{PROP-KEY}).
Key ideas of revealing hidden dissipation of the density and velocity come from analyzing the following linearized system to (\ref{MHD1}):
\begin{equation}\label{explain1}
\left\{
\begin{array}{rl}
&\hspace{-0,5cm}\partial_ta+\textrm{div}u=0,
\smallskip\\
&\hspace{-0,5cm}\partial_t u
+\beta\nabla a=w\cdot \nabla h-\nabla(w\cdot h),
\medskip\\
&\hspace{-0,5cm}\partial_t h-\Delta h
=w\cdot \nabla u-w\textrm{div}u,
\medskip\\
&\hspace{-0,5cm}{\mathop{\rm div}}\, h=0,
\end{array}
\right.
\end{equation}
where $a$ is defined as in (\ref{a}). First, basic energy inequality for (\ref{explain1}) is given by
\begin{equation}\begin{split}\label{explain2}
\frac12\frac{d}{dt}\|\sqrt\beta a,u,h\|^2_N+\|\nabla h\|^2_N=0,
\end{split}\end{equation}
and see Lemma \ref{lemma-pri1} for the corresponding estimate for nonlinear system \eqref{MHD1}.
Next, by $\eqref{explain1}_{1,2}$, we have the following bilateral estimate between $\text{div}\,u$ and $\nabla a$
\begin{equation}\begin{split}\label{explain3}
&\|\textrm{div}\Lambda^{M-r-1}u\|_0^2+\frac{d}{dt}\int\Lambda^{M-r-1}\textrm{div}u\Lambda^{M-r-1}adx\\
=&\beta\|\nabla\Lambda^{M-r-1}a\|_0^2-\int\Delta\Lambda^{M-r-1}h\cdot w\Lambda^{M-r-1}adx,
\end{split}\end{equation}
and see Lemma \ref{lemma-2025-6-27} for the corresponding estimate for the nonlinear system \eqref{MHD1}.
Then, we use $(\ref{explain1})_{2,3}$ to estimate $u$ as
\begin{equation}\label{explain4'}
\begin{split}
\frac12\|&w\cdot\nabla\Lambda^Lu\|_0^2+\frac{d}{dt}\int\Lambda^Lu\cdot(w\cdot\nabla)\Lambda^Lhdx
\leq C\|\mathrm{div}\Lambda^{L}u\|^2_0+C\|h\|_{N+1}^2,
\end{split}
\end{equation}
and see Lemma \ref{lemma-pri3} for the corresponding estimate for the nonlinear system \eqref{MHD1}.
Finally, taking cross product of $w$ with $\eqref{explain1}_2$ and inner product of $w$ with $\eqref{explain1}_3$ gives
\begin{equation}\label{explain4}
\left\{
\begin{array}{rl}
&\hspace{-0,5cm}\partial_t(u\times w)
-\beta(w\times\nabla)a=w\cdot \nabla(h\times w)+(w\times\nabla)(w\cdot h),
\medskip\\
&\hspace{-0,5cm}\partial_t(h\cdot w)-\Delta(h\cdot w)
=\textrm{div}_w(u\times w),
\end{array}
\right.
\end{equation}
from which one gets
$$
\partial_t^2(h\cdot w)-\Delta\partial_t(h\cdot w)-|w|^2\Delta(h\cdot w)-\beta\Delta_wa=0,
$$
where $\textrm{div}_w f:=(w\times\nabla)\cdot f$ and $\Delta_w f:=\textrm{div}_w(w\times\nabla)f.$ This implies that
$w\times\nabla a$ may dissipate. Indeed, one can obtain from \eqref{explain4} that
\begin{equation}
\label{explain5}\begin{split}
\beta\|(w\times\nabla)\Lambda^Ma\|_0^2+&\frac{d}{dt}\int\Big(\textrm{div}_w\Lambda^M(u\times w)\Lambda^Ma+\Lambda^M(h\cdot w)\textrm{div}\Lambda^Mu\Big)dx\\
&\leq C\|h\|^2_{N+1}+\epsilon(\|\textrm{div}\Lambda^{M-r-1}u\|_0^2+\|a\|^2_{M-r}),
\end{split}\end{equation}
and see Lemma \ref{lemma-pri2} for the corresponding estimate for the nonlinear system \eqref{MHD1}.
Combining (\ref{explain3}) with (\ref{explain4'}), together with \eqref{explain5}, and applying the Poincar\'e type
inequality (see Corollary \ref{CorPoincare}) gained from the Diophantine condition, we obtain the following desired dissipation
estimate for $a$ and $u$
\begin{equation*}\begin{split}
 K_3\|u\|_{L-r}^2&+C_1^*\beta\|a\|_{M-r}^2+\frac{d}{dt}\int\bigg[3C_{1}^*\Lambda^{M-r-1}\mathrm{div}u\Lambda^{M-r-1}a
+2(w\cdot\nabla)\Lambda^Lh\cdot\Lambda^Lu\\
&+\frac{8C_1^*}{K_3}\bigg({\rm{div}}_w\Lambda^M(u\times w)\Lambda^Ma+\Lambda^M(h\cdot w)\rm{div}\Lambda^Mu\bigg)\bigg]dx\leq
 C\|h\|^2_{N+1},
\end{split}\end{equation*}
and see Corollary \ref{corollary2} for the corresponding estimate for the nonlinear system \eqref{MHD1}. Combining this dissipative estimate
for $(a, u)$ with that for $h$ in (\ref{explain2}), one can derive the time independent a priori estimate under suitable smallness
condition on the initial data.

The rest of this paper is organized as follows. The next section contains some preliminary lemmas, especially the Poincar\'e type inequalities
gained from the Diophantine condition. Section \ref{sec3} is the main part of this paper, where time independent a priori estimates
are derived for solutions to system \eqref{MHD1}, under some suitable smallness conditions on the initial data. The proof of the main result
of this paper, Theorem \ref{th:main1}, is presented in Section \ref{sec4}.

\section{Preliminaries}\label{sec2}

In this section, we collect some preliminary lemmas which will be used in the next section. We start with the following two Poincar\'e
type inequalities derived from the Diophantine condition, where
the first one is known in the existing literature, see, e.g., Lemma 2.1 in \cite{Xie-Jiu-Liu}, while the second one is new.

\begin{lemma}\label{Diophantine-inequality}
Assume that the constant vector $w\in\mathbb R^3$ satisfies the Diophantine condition \eqref{Diophantine0} with positive constants $c$ and $r$.
Then, for any $s>r$, there is a constant $K_1$ depending only on $|\mathbb T^3|,s,r,c,|w|$ such that
$$\|\Lambda^{s-r} f\|_{0}\leq K_1\|w\cdot\nabla\Lambda^sf\|_{0},\quad \forall f\in H^{s+1}(\mathbb T^3).$$
\end{lemma}

\begin{lemma}\label{Diophantine-inequality0}
Assume that the constant vector $w\in\mathbb R^3$ satisfies the Diophantine condition \eqref{Diophantine0} with positive constants $c$
and $r$. Then, for any $s\geq r$, there is a constant $K_2$ depending only on $|\mathbb T^3|,s,r,c,|w|$ such that
$$
\|f\|_{s-r}\leq K_2\|w\times\Lambda^s\nabla f\|_0,\quad\forall f\in H^{s+1}(\mathbb T^3) \text{ with }\int fdx=0.
$$
\end{lemma}

\begin{proof}
Given any $k:=(k_1,k_2,k_3)\in\mathbb{Z}^3\setminus\{0\}$, without loss of generality, we assume $k_3\neq0$ and denote $\tilde{k}:=(0,-k_3,k_2)$. Then, it holds that
\begin{equation}\begin{split}\label{le-Dio1-4-7}
|(w\cdot\tilde{k})|\leq|w\times k|.
\end{split}\end{equation}
Indeed, direct calculation yields
\begin{equation*}\begin{split}
|w\times k|^2-|(w\cdot\tilde{k})|^2=(w_3k_1-w_1k_3)^2+(w_1k_2-w_2k_1)^2\geq0.
\end{split}\end{equation*}
Thus, combining \eqref{le-Dio1-4-7} with the Diophantine condition \eqref{Diophantine0}, one has
\begin{equation}\begin{split}\label{le-Dio1}
c|k|^{-r}\leq c|\tilde{k}|^{-r}\leq|(\tilde{k}\cdot w)|\leq|w\times k|.
\end{split}\end{equation}
Since $\int fdx=0$ by assumption, one can decompose $f$ as
$$
f=\sum_{k\in\mathbb Z^3\backslash\{0\}}f_ke^{2i\pi k\cdot x}.
$$
As a result, one gets by \eqref{le-Dio1} that
\begin{equation*}\begin{split}
\|f\|_0^2+\|\Lambda^{s-r}f\|^2_0&=\sum_{k\in\mathbb Z^3\backslash\{0\}}\left(1+|2\pi k|^{2(s-r)}\right)|f_k|^2\\
&\leq 2\sum_{k\in\mathbb Z^3\backslash\{0\}} |2\pi k|^{2(s-r)} |f_k|^2\\
&\leq 2(2\pi)^{-2r-2}c^{-2}\sum_{k\in\mathbb Z^3\backslash\{0\}}|2\pi k|^{2s}|2\pi k\times w|^2|f_k|^2\\
&=2(2\pi)^{-2r-2}c^{-2}\|w\times\nabla\Lambda^sf\|_0^2,
\end{split}\end{equation*}
proving the conclusion.
\end{proof}

As a direct corollary of Lemma \ref{Diophantine-inequality} and Lemma \ref{Diophantine-inequality0}, one obtains the following corollary by the Poincar\'e inequality.

\begin{corollary}\label{CorPoincare}
Assume that the constant vector $w\in\mathbb R^3$ satisfies the Diophantine condition \eqref{Diophantine0} with positive constants
$c$ and $r$. Then, for any $s\geq r$, there is a constant $K_3$ depending only on $|\mathbb T^3|,s,r,c,|w|$ such that
$$
K_3\|f\|_{s-r}\leq \|w\times\Lambda^s\nabla f\|_0,\quad K_3\|f\|_{s-r}\leq \|w\cdot\Lambda^s\nabla f\|_0,
$$
for any $f\in H^{s+1}(\mathbb T^3)$ with $\int fdx=0$.
\end{corollary}

The following three lemmas are standard, see, e.g., Lemma 2.4--2.6 in \cite{Wu-Zhai}.

\begin{lemma}\label{Hk-estimate}
Let $s\geq0.$ Then, there is a constant $K_4$ depending only on $s$ and $|\mathbb T^3|$ such that
$$
\|fg\|_{s}\leq K_4\|f\|_{L^\infty}\|g\|_s+K_4\|g\|_{L^\infty}\|f\|_s.
$$
\end{lemma}

\begin{lemma}\label{Hk-estimate-7-8}
Let $s>0.$ Then, there is a constant $K_5$ depending only on $s$ and $|\mathbb T^3|$ such that
$$
\|[\Lambda^s,f]g\|_{0}\leq K_5\|Df\|_{L^\infty}\|\Lambda^{s-1}g\|_{0}+K_5\|g\|_{L^\infty}\|\Lambda^sf\|_0.
$$
\end{lemma}

\begin{lemma}\label{Hk-estimate-1}
Let $s>0$ and $f\in H^s(\mathbb T^3)\cap L^\infty(\mathbb T^3).$ Assume that $F$ is a smooth function on $\mathbb R$ with $F(0)=0$. Then,
it holds that
$$
\|F(f)\|_{s}\leq K_6(1+\|f\|_{L^\infty})^{[s]+1}\|f\|_s,
$$
where the constant $K_6$ depends only on $|\mathbb T^3$, $s$, and $\sup_{k\leq[s]+2, t\leq\|f\|_{L^\infty}}|F^{(k)}(t)|.$
\end{lemma}

The following Poincar\'e type inequality is cited from Danchin--Mucha \cite{Danchin-Mucha} (see Lemma 5.1 there).

\begin{lemma} \label{Poincare-inequality}
Let $f\in L^2(\Omega)$ be a nonnegative and nonzero measurable function on an open and bounded domain $\Omega$ with $\partial\Omega\in C^1$.
Then there is a constant $K_7$ depending only on $\Omega$ such that for any $z\in H^1(\Omega)$, it holds that
$$
\|z\|_{0}\leq\frac{1}{M}\left|\int_{\Omega}fzdx\right|+K_7\left(1+\frac{1}{M}\|M-f\|_{0}\right)\|\nabla z\|_{0}, \quad \text{with}\ M:=\int_{\Omega}fdx.
$$
\end{lemma}

\section{A PRIORI ESTIMATES}
\label{sec3}
This section is devoted to carrying out the a priori estimate for the solutions to system \eqref{MHD1}. To this end, throughout this section,
it is always assumed that $(\rho,u,h)\in C([0,T];H^{N+d}(\mathbb T^3))$ is a smooth solution to system \eqref{MHD1} on
$\mathbb T^3\times[0,T]$ for some positive time $T$.

Denote by $e(\rho)$ the potential energy density, namely
$$
e(\rho)=2\rho\int_{1}^\rho\frac{p(s)-p(1)}{s^2}ds.
$$
Recalling that $p(\rho)$ is increasing in $\rho$, it is clear that $e(\rho)>0$ for $\rho\in(0,1)\cup(1,\infty)$ and $e(1)=0$.
Taking $L^2$ inner product of $\eqref{MHD1}_{2,3}$ with $(u,h)$ and using $\eqref{MHD1}_1$, one gets by direct calculations that
\begin{equation}\begin{split}\label{pri-1}
\frac12\frac{d}{dt}\int\Big(\rho|u|^2dx+|h|^2+e(\rho)\Big)dx+\int|\nabla h|^2dx=0.
\end{split}\end{equation}

Recalling $\int hdx=0$ (see (\ref{junzhi1})), the Poincar\'e inequality gives
\begin{equation}\begin{split}\label{pri-3}
\|h\|_0\leq C\|\nabla h\|_0,
\end{split}\end{equation}
for an absolute positive constant $C$.
Recalling $\int\rho udx=0$ and $\int\rho dx=1$ (see (\ref{junzhi1})), it follows from Lemma \ref{Poincare-inequality} that
\begin{equation}
\begin{split}\label{pri-7-22-1}
\|u\|_0\leq\left|\int \rho udx\right|+K_7\left(1+\|\rho-1\|_{L^2}\right)\|\nabla u\|_{0}\leq K_7(1+\|a\|_{L^2})\|\nabla u\|_0.
\end{split}
\end{equation}
Thanks to (\ref{pri-3}) and (\ref{pri-7-22-1}), one gets by the Poincar\'e inequality that
\begin{equation}
\label{EQUIVN}
\begin{split}
  \|\Lambda^sh\|_0\leq\|h\|_s\leq C_s\|\Lambda^sh\|_0,\quad\|\Lambda^su\|_0\leq\|u\|_s\leq C_s(1+\|a\|_{L^2})\|\Lambda^su\|_0,
\end{split}
\end{equation}
for any nonnegative integer $s$, where $C_s\geq1$ is a constant depending only on $s$.
This fact will be used throughout this section without further mentions.

\begin{lemma}\label{lemma-pri1}Assume that
\begin{equation}\begin{split}\label{pri-2}
\frac12<\rho<\frac32,\quad\mbox{on }\mathbb T^3\times[0,T].
\end{split}\end{equation}
Then, for any $3\leq s\leq N+d$, it holds that
\begin{equation*}
\begin{split}
\frac{d}{dt}\left\|\sqrt{\frac{p'}{\rho}}\Lambda^sa,\sqrt\rho\Lambda^su,\Lambda^sh\right\|^2_0+
\|\Lambda^{s+1}h\|_0^2\leq C(1+\|a,u,h\|_3^2)\|a,u,h\|_3\|a,u,h\|^2_s,
\end{split}\end{equation*}
for a positive constant $C$ depending only on $s, |w|,\beta,$ and $Q$.
\end{lemma}

\begin{proof}
Applying $\Lambda^s$ with $1\leq s\leq N+d$ to $\eqref{MHD1}_{2,3}$ and taking $L^2$ inner product with $\Lambda^su$ and $\Lambda^s h,$ respectively, that is performing $\int [\Lambda^s\eqref{MHD1}_2\cdot \Lambda^s u+\Lambda^s\eqref{MHD1}_3\cdot\Lambda^s h]dx$, then integrating by parts and using $\eqref{MHD1}_{1,4}$ yield
\begin{equation}\begin{split}\label{pri-4}
\frac12\frac{d}{dt}&\|\sqrt\rho\Lambda^su,\Lambda^sh\|^2_0+\|\Lambda^{s+1}h\|_0^2=\int\Big[\Big(-[\Lambda^s,\rho]u_t-[\Lambda^s,\rho u\cdot\nabla]u\\
&-p'\nabla\Lambda^s\rho-[\Lambda^s,p']\nabla\rho+\Lambda^s(h\cdot\nabla h)-\frac12\Lambda^s\nabla(|h|^2)\Big)\cdot\Lambda^s u\\
&-\Big(\Lambda^s(u\cdot\nabla h)-\Lambda^s(h\cdot\nabla u)+\Lambda^s(h\textrm{div}u)\Big)\cdot\Lambda^s h\Big]dx=:\sum_{j=1}^{9}\textrm{I}_j,
\end{split}\end{equation}
where
\begin{align*}
  &\textrm{I}_1:=-\int[\Lambda^s,\rho]u_t\cdot \Lambda^sudx,&&\textrm{I}_2:=-\int[\Lambda^s,\rho u\cdot\nabla]u\cdot \Lambda^sudx,\\
   &\textrm{I}_3:=-\int p'\nabla\Lambda^s\rho\cdot \Lambda^sudx,&&\textrm{I}_4:=-\int [\Lambda^s,p']\nabla\rho\cdot\Lambda^s udx,\\
   &\textrm{I}_5:=\int \Lambda^s(h\cdot\nabla h)\cdot\Lambda^s udx,&&\textrm{I}_6:= -\frac12\int \Lambda^s\nabla(|h|^2)\cdot\Lambda^s udx,\\
   &\textrm{I}_7:=-\int\Lambda^s(u\cdot\nabla h)\cdot\Lambda^shdx,&&\textrm{I}_8:=\int\Lambda^s(h\cdot\nabla u)\cdot\Lambda^shdx,\\ &\textrm{I}_9:=-\int\Lambda^s(h\textrm{div}u)\cdot\Lambda^s hdx.&&
\end{align*}
Estimates on $\textrm{I}_j$ $(j=1,2\cdots9)$ are given as follows.

We first estimate $\textrm{I}_3$.
Using $\eqref{MHD1}_{1}$, one has
\begin{align*}
  \rho\Lambda^s\text{div}u=&\Lambda^s(\rho\text{div}u)-[\Lambda^s,\rho]\text{div}u=-\Lambda^s(\partial_t\rho+u\cdot\nabla\rho)-[\Lambda^s,\rho]
  \text{div}u\\
  =&-(\partial_t\Lambda^s\rho+u\cdot\nabla\Lambda^s\rho)-[\Lambda^s,u\cdot\nabla]\rho-[\Lambda^s,\rho]\text{div}u.
\end{align*}
Noticing that
$$
\left(\frac{p'}{\rho}\right)_t+\textrm{div}\left(\frac{p'}{\rho}u\right)=\left(\frac{2p'}{\rho}-p''\right)\textrm{div}u,
$$
it follows from integrating by parts that
\begin{equation*}\begin{split}
\textrm{I}_3&=\int\Lambda^s\rho\big(p'\textrm{div}\Lambda^su+\nabla p'\cdot \Lambda^su\big)dx\\
&=\int\left[-\frac{p'}{\rho}\Lambda^s\rho\Big(\partial_t\Lambda^s\rho+u\cdot\nabla\Lambda^s\rho
+[\Lambda^s,u\cdot\nabla]\rho+[\Lambda^s,\rho]\textrm{div}u\Big)+\Lambda^s\rho\nabla p'\cdot\Lambda^su\right]dx\\
&=-\frac12\frac{d}{dt}\int\frac{p'}{\rho}|\Lambda^s\rho|^2dx+\frac12\int\left[\left(\frac{p'}{\rho}\right)_t+\textrm{div}
\left(\frac{p'}{\rho}u\right)\right]|\Lambda^s\rho|^2dx+\textrm{I}_{3,b}+\textrm{I}_{3,c}+\textrm{I}_{3,d}\\
&=-\frac12\frac{d}{dt}\int\frac{p'}{\rho}|\Lambda^s\rho|^2dx+\frac12\int\left(\frac{2p'}{\rho}-p''\right)\textrm{div}u|\Lambda^s\rho|^2dx+\textrm{I}_{3,b}+\textrm{I}_{3,c}+\textrm{I}_{3,d}\\
&=-\frac12\frac{d}{dt}\int\frac{p'}{\rho}|\Lambda^sa|^2dx+\textrm{I}_{3,a}+\textrm{I}_{3,b}+\textrm{I}_{3,c}+\textrm{I}_{3,d},
\end{split}\end{equation*}
where
\begin{equation*}\begin{split}
&\textrm{I}_{3,a}:=\frac12\int\left(\frac{2p'}{\rho}-p''\right)\textrm{div}u|\Lambda^s\rho|^2dx,\quad \textrm{I}_{3,b}:=-\int\frac{p'}{\rho}\Lambda^s\rho[\Lambda^s,u\cdot\nabla]\rho dx,\\
&\textrm{I}_{3,c}:=-\int\frac{p'}{\rho}\Lambda^s\rho[\Lambda^s,\rho]\textrm{div}udx,\quad
\textrm{I}_{3,d}:=\int\Lambda^s\rho\nabla p'\cdot\Lambda^su dx.
\end{split}\end{equation*}

Since $\rho\in(\frac12,\frac32)$ and $p(\rho)$ is a smooth function in $\rho$, it follows from Lemma \ref{Hk-estimate-7-8}
and the embedding and Young inequalities that
\begin{align*}
|\textrm{I}_{3,a}|\lesssim&\left\|\frac{2p'}{\rho}-p''\right\|_{L^\infty}\|Du\|_{L^\infty}\|a\|^2_s
\lesssim\|Du\|_{L^\infty}\|a\|^2_s\lesssim\|u\|_{3}\|a\|^2_s,\\
|\textrm{I}_{3,b}|\lesssim&\left\|\frac{p'}{\rho}\right\|_{L^\infty}\|a\|_s\|[\Lambda^s,u\cdot\nabla]\rho\|_0
\lesssim\|a\|_s(\|Da\|_{L^\infty}\|u\|_s+\|a\|_{s}\|Du\|_{L^\infty})\\
\lesssim&\|a,u\|_{3}\|a,u\|^2_s,\\
|\textrm{I}_{3,c}|\lesssim&\left\|\frac{p'}{\rho}\right\|_{L^\infty}\|a\|_s\|[\Lambda^s,\rho]\textrm{div}u\|_0
\lesssim\|a\|_s\|[\Lambda^s,\rho]\textrm{div}u\|_0\\
\lesssim&\|a\|_s(\|Da\|_{L^\infty}\|u\|_s+\|a\|_{s}\|Du\|_{L^\infty})\lesssim\|a,u\|_{3}\|a,u\|^2_s,
\end{align*}
and
$$
|\textrm{I}_{3,d}|\leq\|\nabla p'\|_{L^\infty}\|\Lambda^sa\|_0\|\Lambda^su\|_0\lesssim\|D\rho\|_{L^\infty}\|\Lambda^sa\|_0\|\Lambda^su\|_0\lesssim\|a\|_{3}\|a,u\|^2_s.
$$
Therefore, one has
$$
\left|\textrm{I}_3+\frac12\frac{d}{dt}\int\frac{p'}{\rho}|\Lambda^sa|^2dx\right|\lesssim\|a,u\|_3\|u,u\|_s^2.
$$

For $\textrm{I}_2$ and $\textrm{I}_4,$ similar to the estimates on $\textrm{I}_{3,a}$--$\textrm{I}_{3,d}$, it follows from
Lemma \ref{Hk-estimate}, Lemma \ref{Hk-estimate-7-8}, Lemma \ref{Hk-estimate-1}, \eqref{pri-2}, and the embedding and Young inequalities
that
\begin{align*}
|\textrm{I}_{2}|=&\left|\int[\Lambda^s,\rho u\cdot\nabla]u\cdot\Lambda^sudx\right|
\leq\|[\Lambda^s,\rho u\cdot\nabla]u\|_0\|\Lambda^su\|_0\\
\lesssim&(\|D(\rho u)\|_{L^\infty}\|u\|_s+\|\rho u\|_s\|Du\|_{L^\infty})\|u\|_s\\
\lesssim& (\|Da\|_{L^\infty}\|u\|_{L^\infty}+\|\rho\|_{L^\infty}\|Du\|_{L^\infty})\|u\|_s^2\\
&+[\|\rho\|_{L^\infty}\|u\|_s+(1+\|a\|_s)\|u\|_{L^\infty}]\|Du\|_{L^\infty} \|u\|_s\\
\lesssim&(1+\|u\|_2)\|a,u\|_3\|a, u\|^2_s+(1+\|a\|_s)\|u\|_2\|u\|_3\|u\|_s\\
\lesssim&(1+\|u\|_2)\|a,u\|_3\|a,u\|_s^2
\end{align*}
and
\begin{align*}
|\textrm{I}_{4}|=&\left|\int[\Lambda^s,p']\nabla\rho\cdot\Lambda^sudx\right|
\leq\|[\Lambda^s,p']\nabla\rho\|_0\|\Lambda^su\|_0\\
\lesssim&(\|\nabla p'\|_{L^\infty}\|\Lambda^sa\|_0+\|\Lambda^sp'\|_{0}\|Da\|_{L^\infty})\|u\|_s\\
\lesssim&(\|\nabla a\|_{L^\infty}\|\Lambda^sa\|_0+\|\Lambda^s(p'(1+a)-p'(1))\|_{0}\|Da\|_{L^\infty})\|u\|_s\\
\lesssim&\|a\|_{s}\|Da\|_{L^\infty}\|u\|_s\lesssim\|a\|_{3}\|a,u\|^2_s.
\end{align*}

We now deal with $\textrm{I}_{5}$--$\textrm{I}_{9}.$ It follows from Lemma \ref{Hk-estimate}, \eqref{EQUIVN}, and the embedding and
Young inequalities that
\begin{equation*}\begin{split}
|\textrm{I}_{5}|&=\left|\int\Lambda^s(h\cdot\nabla h)\cdot\Lambda^s udx\right|
\leq\|\Lambda^s(h\cdot\nabla h)\|_0\|\Lambda^su\|_0\\
&\lesssim(\|h\|_{L^\infty}\|\nabla h\|_s+\|\nabla h\|_{L^\infty}\|h\|_s)\|u\|_s\lesssim\|h\|_{3}\|h\|_{s+1}\|u\|_s,\\
&\lesssim\|h\|_{3}\|\Lambda^{s+1}h\|_0\|u\|_s\leq C_\epsilon\|h\|^2_{3}\|u\|^2_s+\epsilon\|\Lambda^{s+1}h\|^2_0\\
\end{split}\end{equation*}
and similarly
\begin{align*}
  |I_6|=&\frac12\left|\int\Lambda^s\nabla|h|^2\cdot\Lambda^sudx\right|=\left|\int\Lambda^s((\nabla h)^Th)\cdot\Lambda^sudx\right|\\
  \lesssim&\epsilon\|\Lambda^{s+1}h\|_0^2+C_\epsilon\|h\|_3^2\|u\|_s^2,
\end{align*}
for any $\epsilon>0$. For $I_7$ and $I_8$, integrating by parts, applying Lemma \ref{Hk-estimate}, and using the embedding and
Young inequalities, one deduces for any $\epsilon>0$ that
\begin{equation*}\begin{split}
|\textrm{I}_{7}|&=\left|\int\Lambda^s(u\cdot\nabla h)\cdot\Lambda^s hdx\right|=\left|\int\Lambda^{s-1}(u\cdot\nabla h)\cdot\Lambda^{s+1} hdx\right|\\
&\lesssim\|u,Dh\|_{L^\infty}\|u,Dh\|_{s-1}\|\Lambda^{s+1}h\|_0\leq C_\epsilon\|u,h\|^2_{3}\|u,h\|^2_{s}+\epsilon\|\Lambda^{s+1}h\|^2_0,\\
|\textrm{I}_{8}|&=\left|\int\Lambda^s(h\cdot\nabla u)\cdot\Lambda^s hdx\right|=\left|\int\Lambda^{s-1}(h\cdot\nabla u)\cdot\Lambda^{s+1} hdx\right|\\
&\lesssim\|Du,h\|_{L^\infty}\|Du,h\|_{s-1}\|\Lambda^{s+1}h\|_0\leq C_\epsilon\|u,h\|^2_{3}\|u,h\|^2_s+\epsilon\|\Lambda^{s+1}h\|^2_0,
\end{split}\end{equation*}
and similarly
$$
|\textrm{I}_{9}|=\left|\int\Lambda^s(h\textrm{div}u)\cdot\Lambda^s hdx\right|\leq C_\epsilon\|u,h\|^2_{3}\|u,h\|^2_s+\epsilon\|\Lambda^{s+1}h\|^2_0.
$$

It remains to estimate $\textrm{I}_{1}.$ To this end, we first estimate $\|u_t\|_{L^\infty}.$ One deduces by \eqref{pri-2}
and the bedding inequality that
\begin{equation}\begin{split}\label{pri-5}
\|u_t\|_{L^\infty}&=\left\|-(u\cdot\nabla)u-\frac{p'}{\rho}\nabla a+\rho^{-1}(H\cdot\nabla)h-\rho^{-1}(\nabla h)^{\top}H\right\|_{L^\infty}\\
&\lesssim\|u\|_{L^\infty}\|Du\|_{L^\infty}+\|Da\|_{L^\infty}+(1+\|h\|_{L^\infty})\|Dh\|_{L^\infty}\\
&\lesssim(1+\|u,h\|_{2})\|a,u,h\|_3.
\end{split}\end{equation}
Next, we estimate $\|u_t\|_{s-1}.$ Note that \eqref{pri-2} implies $a\in(-\frac12,\frac12)$. Thus, it follows from Lemma \ref {Hk-estimate-1}
that
\begin{equation*}\begin{split}
\left\|\frac{p'}{\rho}\right\|_{s-1}&=\left\|\frac{p'(a+1)}{a+1}-p'(1)+p'(1)\right\|_{s-1}\lesssim\left\|\frac{p'(a+1)}{a+1}-p'(1)\right\|_{s-1}+1\\
&\lesssim\|a\|_{s-1}+1
\end{split}\end{equation*}
and
\begin{equation*}\begin{split}
\left\|\frac{1}{\rho}\right\|_{s-1}&=\|(a+1)^{-1}-1+1\|_{s-1}\lesssim\|(a+1)^{-1}-1\|_{s-1}+1\\
&\lesssim\|a\|_{s-1}+1.
\end{split}\end{equation*}
Thanks to the above two estimates, recalling $s\geq3$ and $\rho\in(\frac12,\frac32)$, it follows from Lemma \ref{Hk-estimate} and the Young inequality that
\begin{equation}\begin{split}\label{pri-6}
\|u_t\|_{s-1}=&\left\|-u\cdot\nabla u-\frac{p'}{\rho}\nabla a+\rho^{-1}H\cdot\nabla h-\rho^{-1}(\nabla h)^{\top}H\right\|_{s-1}\\
\lesssim&\|u\|_{L^\infty}\|Du\|_{s-1}+\|u\|_{s-1}\|Du\|_{L^\infty}+\left\|\frac{p'}{\rho}\right\|_{s-1}\|Da\|_{L^\infty}\\
&+\left\|\frac{p'}{\rho}\right\|_{L^\infty}\|Da\|_{s-1}+\left\|\frac1\rho\right\|_{L^\infty}
\|H\cdot \nabla h, (\nabla h)^\top H\|_{s-1}\\
&+\left\|\frac1\rho\right\|_{s-1}
\|H\cdot\nabla h, (\nabla h)^\top H\|_{L^\infty}\\
\lesssim&\|u\|_3\|u\|_s+(1+\|a\|_s)\|a\|_3+\|a\|_s+\|h+w\|_{L^\infty}\|\nabla h\|_{s-1}\\
&+\|h+w\|_{s-1}\|\nabla h\|_{L^\infty}
+(1+\|a\|_s)\|\nabla h\|_{L^\infty}\|h+w\|_{L^\infty}\\
\lesssim&\|u\|_3\|u\|_s+(1+\|a\|_s)\|a\|_3+\|a\|_s+(1+\|h\|_2)\|h\|_{s}\\
&+(1+\|h\|_{s})\|h\|_3
+(1+\|a\|_s)\|h\|_3(1+\|h\|_3)\\
\lesssim&(1+\|a,u,h\|_3^2)\|a,u,h\|_s.
\end{split}\end{equation}
This together with \eqref{pri-5}, Lemma \ref{Hk-estimate-7-8}, and the Young inequality yields
\begin{equation*}\begin{split}
|\textrm{I}_{1}|=&\left|\int[\Lambda^s,\rho]u_t\cdot\Lambda^s udx\right|
=\left|\int[\Lambda^s,a]u_t\cdot\Lambda^s udx\right|\\
\leq&\|[\Lambda^s,a]u_t\|_0\|\Lambda^s u\|_0\lesssim
(\|a\|_s\|u_t\|_{L^\infty}+\|Da\|_{L^\infty}\|u_t\|_{s-1})\|u\|_s\\
\lesssim&\|a\|_s(1+\|u,h\|_2)\|a,u,h\|_3\|u\|_s+
\|a\|_3(1+\|a,u,h\|_3^2)\|a,u,h\|_s\|u\|_s\\
\lesssim&(1+\|a,u,h\|_3^2)\|a,u,h\|_3\|a,u,h\|^2_s.
\end{split}\end{equation*}

Plugging the estimates for $\textrm{I}_{1}$--$\textrm{I}_{9}$ into \eqref{pri-4} and choosing $\epsilon$ sufficiently small lead to the conclusion.
\end{proof}

Due to the lack of dissipation terms for the density and velocity field, one needs to
exploit the hidden dissipation benefit from the background magnetic
field. To this end, we rewrite system $\eqref{MHD1}$ as
\begin{equation}\label{MHD2}
\left\{
\begin{array}{rl}
&\hspace{-0,5cm}\partial_ta+\textrm{div}u=R_1,
\smallskip\\
&\hspace{-0,5cm}\partial_t u
+\beta\nabla a=(w\cdot \nabla)h-\nabla(w\cdot h)+R_2,
\medskip\\
&\hspace{-0,5cm}\partial_t h-\Delta h
=(w\cdot \nabla)u-w\textrm{div}u+R_3,
\medskip\\
&\hspace{-0,5cm}{\mathop{\rm div}}\, h=0,
\medskip\\
&\hspace{-0,5cm}(\rho,u,h)_{|t=0}=(\rho_0,u_0,h_0),
\end{array}
\right.
\end{equation}
where $\beta:=p'(1)>0$, and $R_1,R_2,R_3$ are the nonlinear terms expressed as
\begin{equation*}\begin{split}
&R_1:=-\textrm{div}(a u),\\
&R_2:=-a\partial_tu-(p'(a+1)-p'(1))\nabla a-\rho(u\cdot\nabla)u+(h\cdot\nabla)h-\nabla\left(\frac{|h|^2}{2}\right),\\
&R_3:=-(u\cdot\nabla)h+(h\cdot\nabla)u-h\textrm{div}u.
\end{split}\end{equation*}

Estimates on $R_1,R_2$ and $R_3$ are given in the following lemma.

\begin{lemma}\label{lemma-2025-7-22}
It holds that for $l\geq0$ that
\begin{equation*}
\|R_1\|_l\leq C\|a,u\|_{2}\|a,u\|_{l+1},\quad
\|R_3\|_l\leq C\|u,h\|_3\|u,h\|_{l+1},
\end{equation*}
for a positive constant $C$ depending only on $l$.
Assuming further that \eqref{pri-2} holds, then for $2\leq m\leq N+d$, it holds that
\begin{equation*}
\|R_2\|_m\leq C(1+\|a,u,h\|_3^2)\|a,u,h\|_3\|a,u,h\|_{m+1},
\end{equation*}
for a positive constant $C$ depending only on $m, |w|,\beta$, and $Q$.
\end{lemma}

\begin{proof}It follows from Lemma \ref{Hk-estimate} and $\|\cdot\|_{L^\infty}\lesssim\|\cdot\|_2$ that for $l\geq0$
\begin{equation*}\begin{split}
\|R_1\|_l\leq\|au\|_{l+1}\lesssim\|a,u\|_{L^\infty}\|a,u\|_{l+1}\lesssim\|a,u\|_{2}\|a,u\|_{l+1}
\end{split}\end{equation*}
and
\begin{equation*}\begin{split}
\|R_3\|_l\leq&\|(u\cdot\nabla)h\|_l+\|(h\cdot\nabla)u\|_l+\|h\textrm{div}u\|_{l}\\
\lesssim&\|u,\nabla h\|_{L^\infty}\|u,\nabla h\|_l+\|h,\nabla u\|_{L^\infty}\|h,\nabla u\|_l\lesssim\|u,h\|_3\|u,h\|_{l+1},
\end{split}\end{equation*}
proving the estimates for $R_1$ and $R_3$.
It remains to estimate $R_{2}.$ One deduces by \eqref{pri-5}, \eqref{pri-6}, Lemma \ref{Hk-estimate},
and the embedding and Young inequalities that
\begin{equation*}
\begin{split}
\|au_t\|_m\lesssim&\|a\|_{L^\infty}\|u_t\|_m+\|a\|_m\|u_t\|_{L^\infty}\\
\lesssim& \|a\|_3(1+\|a,u,h\|_3^2)\|a,u,h\|_{m+1}+\|a\|_m(1+\|u,h\|_{3})\|a,u,h\|_3\\
\lesssim& (1+\|a,u,h\|_3^2)\|a,u,h\|_3\|a,u,h\|_{m+1}.
\end{split}
\end{equation*}
It follows from \eqref{pri-2}, Lemma \ref{Hk-estimate-1}, and the embedding inequality that
\begin{equation*}\begin{split}
&\|(p'(a+1)-p'(1))\nabla a\|_m\\
\lesssim&\|p'(a+1)-p'(1)\|_{L^\infty}\|a\|_{m+1}+\|p'(a+1)-p'(1)\|_{m}\|\nabla a\|_{L^\infty}\\
\lesssim&\|a\|_{L^\infty}\|a\|_{m+1}+\|a\|_{m}\|\nabla a\|_{L^\infty}\lesssim\|a\|_3\|a\|_{m+1}.
\end{split}\end{equation*}
Similarly, it follows from  Lemma \ref{Hk-estimate} that
\begin{equation*}\begin{split}
&\left\|-\rho(u\cdot\nabla)u+(h\cdot\nabla)h-\nabla\left(\frac{|h|^2}{2}\right)\right\|_m\\
\lesssim& \|\rho u,Du\|_{L^\infty}\|\rho u,Du\|_{m}+\|h,Dh\|_{L^\infty}\|h,Dh\|_{m}\\
\lesssim&(\|\rho\|_{L^\infty}\|u\|_{L^\infty}+\|Du\|_{L^\infty})(\|\rho\|_{L^\infty}\|u\|_m+\|\rho\|_m\|u\|_{L^\infty}+\|u\|_{m+1})
+\|h\|_3\|h\|_{m+1}\\
\lesssim&\|u\|_3(\|u\|_{m+1}+\|\rho\|_m\|u\|_{3})+\|h\|_3\|h\|_{m+1}\\
\lesssim&\|a,u,h\|_3\|a,u,h\|_{m+1}+\|u\|_3^2(1+\|a\|_{m+1})\\
\lesssim&(1+\|a,u,h\|_3)\|a,u,h\|_3\|a,u,h\|_{m+1}.
\end{split}\end{equation*}
Combining the above three inequalities yields the estimate for $R_2$.
\end{proof}

The following lemma stating the mutual control between $\text{div}u$ and $\nabla a$ is due to the Euler structure
of subsystem $\eqref{MHD2}_{1,2}$.

\begin{lemma}\label{lemma-2025-6-27}
Suppose \eqref{pri-2} holds. Then, it holds for any $t\in[0,T]$ that
\begin{align*}
  \left|\|\Lambda^{M-r-1}\rm{div}u\|_0^2-\beta\|\Lambda^{M-r-1}\nabla a\|_0^2
  +\frac{d}{dt}\int\Lambda^{M-r-1}\rm{div}u\Lambda^{M-r-1}adx\right|\\
  \leq \|h\|_{M-r}\|a\|_{M-r}+C(1+\|a,u,h\|_3^2)\|a,u,h\|_{3}\|a,u,h\|^2_{N}
\end{align*}
and in particular
\begin{equation*}\begin{split}
&\|\mathrm{div}\Lambda^{M-r-1}u\|_0^2+\frac{d}{dt}\int\Lambda^{M-r-1}\mathrm{div}u\Lambda^{M-r-1}adx\\
\leq&\beta\|a\|_{M-r}^2
+\|h\|_{M-r}\|a\|_{M-r}+C(1+\|a,u,h\|_3^2)\|a,u,h\|_{3}\|a,u,h\|^2_{N},
\end{split}\end{equation*}
for a positive constant $C$ depending only on $N, |w|,\beta,$ and $Q$.
\end{lemma}

\begin{proof}
Applying $\Lambda^{M-r-1}$ to $\eqref{MHD2}_{2}$, taking $L^2$ inner product to the resultant with $\nabla\Lambda^{M-r-1}a$,
integrating by parts, and using $\eqref{MHD2}_{1}$ yield
\begin{equation}\begin{split}\label{2-pri-7}
\beta&\|\nabla\Lambda^{M-r-1}a\|_0^2
=\int \Lambda^{M-r-1}\textrm{div}u_t\Lambda^{M-r-1}adx\\
&-\int \nabla\Lambda^{M-r-1}a\cdot\nabla\Lambda^{M-r-1}(h\cdot w) dx-\int \textrm{div}\Lambda^{M-r-1}R_2\Lambda^{M-r-1}a dx.
\end{split}\end{equation}
Thanks to $\eqref{MHD2}_1$, the first term on the right-hand side of \eqref{2-pri-7} can be rewritten as
\begin{equation*}\begin{split}
&\int\Lambda^{M-r-1}\textrm{div}u_t\Lambda^{M-r-1}adx\\
=&\frac{d}{dt}\int\Lambda^{M-r-1}\textrm{div}u\Lambda^{M-r-1}adx-\int\Lambda^{M-r-1}\textrm{div}u\Lambda^{M-r-1}a_tdx\\
=&\frac{d}{dt}\int\Lambda^{M-r-1}\textrm{div}u\Lambda^{M-r-1}adx+\|\textrm{div}\Lambda^{M-r-1}u\|_0^2\\
&-\int\Lambda^{M-r-1}\textrm{div}u\Lambda^{M-r-1}R_1dx.
\end{split}\end{equation*}
Thus, inserting the above equality into \eqref{2-pri-7} leads to
\begin{equation}\begin{split}\label{2-pri-8}
&\|\textrm{div}\Lambda^{M-r-1}u\|_0^2+\frac{d}{dt}\int\Lambda^{M-r-1}\textrm{div}u\Lambda^{M-r-1}adx\\
=&\beta\|\nabla\Lambda^{M-r-1}a\|_0^2+\int \nabla\Lambda^{M-r-1}a\cdot\nabla\Lambda^{M-r-1}(h\cdot w) dx\\
&+\int \textrm{div}\Lambda^{M-r-1}R_2\Lambda^{M-r-1}a dx+\int\Lambda^{M-r-1}\textrm{div}u\Lambda^{M-r-1}R_1dx.
\end{split}\end{equation}
By Lemma \ref{lemma-2025-7-22} and since $M\leq N-r-2$, one deduces
$$
\left|\int \nabla\Lambda^{M-r-1}a\cdot\nabla\Lambda^{M-r-1}(h\cdot w) dx\right|\leq\|h\|_{M-r}\|a\|_{M-r}
$$
and
\begin{align*}
&\left|\int\textrm{div}\Lambda^{M-r-1}R_2\Lambda^{M-r-1}adx\right|\\
\lesssim&\|R_2\|_{M-r}\|\Lambda^{M-r-1}a\|_0\lesssim(1+\|a,u,h\|_3^2)\|a,u,h\|_{3}\|a,u,h\|^2_{N},\\
&\left|\int\Lambda^{M-r-1}\textrm{div}u\Lambda^{M-r-1}R_1dx\right|\\
\lesssim&\|R_1\|_{M-r-1}\|\Lambda^{M-r}u\|_0\lesssim\|a,u,h\|_{3}\|a,u,h\|^2_{N}.
\end{align*}
Inserting the above estimates into \eqref{2-pri-8} yields the desire  conclusion.
\end{proof}

The follow lemma is derived by using $\eqref{MHD2}_{2,3,4}$.

\begin{lemma}\label{lemma-pri3}
Suppose \eqref{pri-2} holds. Then, for any $t\in[0,T]$, it follows that
\begin{equation*}\begin{split}
&\frac12\|w\cdot\nabla\Lambda^Lu\|_0^2+\frac{d}{dt}\int(w\cdot\nabla)\Lambda^Lh\cdot \Lambda^Ludx\\
\leq& C\|\mathrm{div}\Lambda^Lu\|^2_0+C\|h\|^2_{N+1}+C(1+\|a,u,h\|_3^2)\|a,u,h\|_3\|a,u,h\|^2_{N},
\end{split}\end{equation*}
where $C$ is a positive constant depending only on $N,|w|,\beta,$ and $Q$.
\end{lemma}
\begin{proof}
Applying $\Lambda^L$ with to $\eqref{MHD2}_{3}$ and taking $L^2$ inner product with $w\cdot\nabla\Lambda^Lu$ yield
\begin{equation}\begin{split}\label{pri-16}
\|w\cdot\nabla\Lambda^Lu\|_0^2=\int(w\cdot\nabla)\Lambda^Lu\cdot\Lambda^L(h_t-\Delta h+\textrm{div}uw-R_3)dx.
\end{split}\end{equation}
For the first term on the right-hand side of \eqref{pri-16}, one gets by $\eqref{MHD2}_2$ that
\begin{equation}\begin{split}\label{pri-17}
&\int(w\cdot\nabla)\Lambda^Lu\cdot\Lambda^Lh_tdx=-\int(w\cdot\nabla)\Lambda^Lh_t\cdot \Lambda^Ludx\\
&\ =-\frac{d}{dt}\int(w\cdot\nabla)\Lambda^Lh\cdot \Lambda^Ludx+\int(w\cdot\nabla)\Lambda^Lh\cdot\Lambda^Lu_t dx\\
&\ =-\frac{d}{dt}\int(w\cdot\nabla)\Lambda^Lh\cdot \Lambda^Ludx+\int(w\cdot\nabla)\Lambda^Lh\cdot\Lambda^L(w\cdot\nabla h+R_2)dx\\
&\ =-\frac{d}{dt}\int(w\cdot\nabla)\Lambda^Lh\cdot\Lambda^Lu dx+\textrm{III}_1+\textrm{III}_2,
\end{split}\end{equation}
where
$$
\textrm{III}_1:=\int(w\cdot\nabla)\Lambda^Lh\cdot\Lambda^L(w\cdot\nabla h) dx,
\quad \textrm{III}_2:=\int(w\cdot\nabla)\Lambda^Lh\cdot\Lambda^L R_2 dx,
$$
and the following has been used
$$
-\int(w\cdot\nabla)\Lambda^Lh\cdot\nabla\Lambda^L(\beta a+ h\cdot w)dx=\int(w\cdot\nabla)\Lambda^L\text{div}h\Lambda^L(\beta a+ h\cdot w)dx=0,
$$
as $\text{div}h=0$.
Substituting \eqref{pri-17} into \eqref{pri-16} yields
\begin{equation}
\begin{split}\label{pri-19}
\|&w\cdot\nabla\Lambda^Lu\|_0^2+\frac{d}{dt}\int\Lambda^Lu\cdot(w\cdot\nabla)\Lambda^Lhdx=\sum_{j=1}^5\textrm{III}_j,
\end{split}
\end{equation}
where $\textrm{III}_1$ and $\textrm{III}_2$ are as above, and
\begin{equation*}\begin{split}
&\textrm{III}_3:=-\int(w\cdot\nabla)\Lambda^Lu\cdot\Lambda^L\Delta hdx,\quad \textrm{III}_4:=\int(w\cdot\nabla)\Lambda^L(u\cdot w)\textrm{div}\Lambda^Ludx,\\
&\textrm{III}_5:=-\int (w\cdot\nabla)\Lambda^Lu\cdot\Lambda^LR_3dx.
\end{split}\end{equation*}
For $\textrm{III}_1, \textrm{III}_3,$ and $\textrm{III}_4$, noticing that $L\leq N-1$, one deduces
\begin{align*}
  |\textrm{III}_1|=&\|(w\cdot\nabla)\Lambda^Lh\|_0^2\lesssim\|h\|_{L+1}^2\lesssim\|h\|^2_{N+1},\\
  |\textrm{III}_3|=&\left|\int(w\cdot\nabla)\Lambda^Lu\cdot\Lambda^L\Delta hdx\right|\lesssim\|(w\cdot\nabla)\Lambda^Lu\|_0\|h\|_{L+2}\\
  \leq&C_\epsilon\|h\|^2_{N+1}+\epsilon\|(w\cdot\nabla)\Lambda^Lu\|_0^2,
\end{align*}
and
\begin{align*}
  |\textrm{III}_4|=&\left|\int(w\cdot\nabla)\Lambda^L(u\cdot w)\textrm{div}\Lambda^Ludx\right|
  \leq\|w\cdot\nabla\Lambda^L(u\cdot w)\|_0\|\textrm{div}\Lambda^Lu\|_0\\
\leq& C_{\epsilon}\|\textrm{div}\Lambda^Lu\|^2_0+\epsilon\|(w\cdot\nabla)\Lambda^Lu\|^2_0,
\end{align*}
for any $\epsilon>0$. For $\textrm{III}_2$ and $\textrm{III}_5$, it follows lemma \ref{lemma-2025-7-22} and $L\leq M-r-1$ that
\begin{equation*}
\begin{split}
|\textrm{III}_2|=&\left|\int(w\cdot\nabla)\Lambda^Lh\cdot\Lambda^L R_2 dx\right|\lesssim \|h\|_{L+1}\|\Lambda^LR_2\|_0\\
\lesssim&(1+\|a,u,h\|_3^2)\|a,u,h\|_3\|a,u,h\|^2_{N}
\end{split}
\end{equation*}
and
\begin{equation*}\begin{split}
|\textrm{III}_5|=&\left|\int (w\cdot\nabla)\Lambda^Lu\cdot\Lambda^LR_3dx\right|
\lesssim\|u\|_{L+1}\|\Lambda^LR_3\|_{0}\\
\lesssim&\|u\|_{L+1}\|u,h\|_3\|u,h\|_{N}\lesssim\|a,u,h\|_3\|a,u,h\|^2_{N}.
\end{split}\end{equation*}
Thus, by inserting the estimates of $\textrm{III}_1-\textrm{III}_5$ into \eqref{pri-19} and taking $\epsilon$ small enough,
the conclusion follows.
\end{proof}

For brevity, the following notations will be adopted
$$
\textrm{div}_w F:=(w\times\nabla)\cdot F,\quad \Delta_w f:=(w\times\nabla)\cdot(w\times\nabla)f.
$$
By taking the cross product of $w$ with equation $\eqref{MHD2}_2$ and the inner product of $w$ with equation $\eqref{MHD2}_3$ one has
\begin{equation}\label{MHD3}
\left\{
\begin{array}{rl}
&\hspace{-0,5cm}\partial_ta+\textrm{div}u=R_1,
\smallskip\\
&\hspace{-0,5cm}\partial_t(u\times w)
-\beta(w\times\nabla)a=(w\cdot\nabla)(h\times w)+(w\times\nabla)(w\cdot h)+R_2\times w,
\medskip\\
&\hspace{-0,5cm}\partial_t(h\cdot w)-\Delta(h\cdot w)
=\textrm{div}_w(u\times w)+R_3\cdot w,
\medskip\\
&\hspace{-0,5cm}{\mathop{\rm div}}\, h=0.
\end{array}
\right.
\end{equation}
Note that we have used the following identity in deriving the third equation
$$
\textrm{div}_w(u\times w)=(w\cdot\nabla)(u\cdot w)-|w|^2\textrm{div}u.
$$

The following lemma exploits a new dissipating mechanism for the perturbed density $a$ which is generated from the interaction
between the velocity field and the constant background magnetic field $w$.

\begin{lemma}\label{lemma-pri2}Suppose \eqref{pri-2} holds. Then, it holds for any $t\in[0,T]$ that
\begin{equation*}\begin{split}
\beta\|(w\times\nabla)\Lambda^Ma\|_0^2+&\frac{d}{dt}\int\Big(\mathrm{div}_w\Lambda^M(u\times w)\Lambda^Ma+\Lambda^M(h\cdot w)\mathrm{div}\Lambda^Mu\Big)dx\\
\leq& C(1+\|a,u,h\|_{3}^2)\|a,u,h\|_{3}\|a,u,h\|^2_{N}+ C_\epsilon\|h\|^2_{N+1}\\
&+\epsilon(\|a\|^2_{M-r}+\|\mathrm{div}\Lambda^{M-r-1}u\|^2_0),
\end{split}\end{equation*}
for any $\epsilon>0$, where $C$ and $C_\epsilon$ are positive constants depending only on $N, |w|,\beta,$ and $Q$ (with $C_\epsilon$ depending also on $\epsilon$).
\end{lemma}
\begin{proof}
Applying $\Lambda^M$ with to $\eqref{MHD3}_{2}$, taking $L^2$ inner product to the resultant with $(w\times\nabla)\Lambda^Ma$, and integrating by parts yield
\begin{equation}\begin{split}\label{pri-7}
\beta\|(w\times\nabla)\Lambda^Ma\|_0^2&=\int\Big(-\textrm{div}_w\Lambda^M(u_t\times w)+(w\cdot\nabla)^2(\Lambda^Mh\cdot w)\\
&+\Delta_w\Lambda^M(h\cdot w)+\textrm{div}_w\Lambda^M(R_2\times w)\Big)\Lambda^Madx,
\end{split}\end{equation}
where we have used
$$
w\cdot\nabla\textrm{div}_w(h\times w)=w\cdot\nabla(w\cdot\nabla(h\cdot w)-|w|^2\textrm{div}h)=(w\cdot\nabla)^2(h\cdot w),
$$
as $\textrm{div}h=0$.
Using $\eqref{MHD3}_{1,3}$, one has
\begin{align}\label{811}
&-\int\textrm{div}_w\Lambda^M(u_t\times w)\Lambda^Madx\nonumber\\
=&-\frac{d}{dt}\int\textrm{div}_w\Lambda^M(u\times w)\Lambda^Madx+\int\textrm{div}_w\Lambda^M(u\times w)\Lambda^Ma_tdx\nonumber\\
=&-\frac{d}{dt}\int\textrm{div}_w\Lambda^M(u\times w)\Lambda^Madx-\int\textrm{div}_w\Lambda^M(u\times w)\textrm{div}\Lambda^Mudx+\textrm{II}_1\nonumber\\
=&-\frac{d}{dt}\int\textrm{div}_w\Lambda^M(u\times w)\Lambda^Madx-\int\Lambda^M\partial_t(h\cdot w)\textrm{div}\Lambda^Mudx+\textrm{II}_1+\textrm{II}_2+\textrm{II}_3\nonumber\\
=&-\frac{d}{dt}\int\big(\textrm{div}_w\Lambda^M(u\times w)\Lambda^Ma+\Lambda^M(h\cdot w)\textrm{div}\Lambda^Mu\big)dx+\sum_{j=1}^4\textrm{II}_j,
\end{align}
where
\begin{equation*}\begin{split}
&\textrm{II}_1:=\int\textrm{div}_w\Lambda^M(u\times w)\Lambda^MR_1dx,\quad \textrm{II}_2:=-\int\Lambda^{M+2}(h\cdot w)\textrm{div}\Lambda^Mudx,\\ &\textrm{II}_3:=\int\Lambda^{M}(R_3\cdot w)\textrm{div}\Lambda^Mudx,\quad \textrm{II}_4:=\int\Lambda^M(h\cdot w)\textrm{div}\Lambda^Mu_tdx.
\end{split}\end{equation*}
Thus, inserting (\ref{811}) into \eqref{pri-7} leads to
\begin{equation}\begin{split}\label{pri-8}
\beta\|(w\times\nabla)\Lambda^Ma\|_0^2&+\frac{d}{dt}\int\Big(\textrm{div}_w\Lambda^M(u\times w)\Lambda^Ma+\Lambda^M(h\cdot w)\textrm{div}\Lambda^Mu\Big)dx\\
=\sum_{j=1}^7\textrm{II}_j,
\end{split}\end{equation}
where $\textrm{II}_1, \textrm{II}_2, \textrm{II}_3, \textrm{II}_4$ are as above and
\begin{equation*}\begin{split}
&\textrm{II}_5:=\int(w\cdot\nabla)^2(\Lambda^Mh\cdot w)\Lambda^Madx,\quad\textrm{II}_6:=\int\Delta_w\Lambda^M(h\cdot w)\Lambda^Madx,\\ &\textrm{II}_7:=\int\textrm{div}_w\Lambda^M(R_2\times w)\Lambda^Madx.
\end{split}\end{equation*}
For $\textrm{II}_1, \textrm{II}_3,$ and $\textrm{II}_7$, one deduces by Lemma \ref{lemma-2025-7-22} and $3\leq M+1<N$ that
\begin{equation*}\begin{split}
|\textrm{II}_1|&\lesssim\|u\|_{M+1}\|R_1\|_{M}\lesssim\|u\|_{N}\|a,u\|_{2}\|a,u\|_{N},\\
|\textrm{II}_3|&\lesssim\|u\|_{M+1}\|R_3\|_M\lesssim\|u\|_{N}\|u,h\|_3\|u,h\|_{N},\\
\end{split}\end{equation*}
and
\begin{equation*}\begin{split}
|\textrm{II}_7|&\lesssim\|R_2\|_{M+1}\|a\|_{M}\lesssim(1+\|a,u,h\|_{3}^2)\|a,u,h\|_{3}\|a,u,h\|^2_{N}.
\end{split}\end{equation*}
For $\textrm{II}_2$, it follows from integrating by parts and $M+r+2\leq N$ that
\begin{equation*}\begin{split}
|\textrm{II}_2|&=\Big|\int\Lambda^{M+r+3}(h\cdot w)\textrm{div}\Lambda^{M-r-1}udx\Big|\lesssim\|h\|_{N+1}\|\textrm{div}\Lambda^{M-r-1}u\|_0\\
&\leq C_\epsilon\|h\|^2_{N+1}+\epsilon\|\textrm{div}\Lambda^{M-r-1}u\|^2_0.\\
\end{split}\end{equation*}
For $\textrm{II}_4$, it follows from $\eqref{MHD2}_{2}$ that
$$
\textrm{II}_4=\int\Lambda^M(h\cdot w)\textrm{div}\Lambda^M(-\beta\nabla a-\nabla(w\cdot h)+R_2)dx=:\sum_{j=1}^3\textrm{II}_{4,j},
$$
where recalling $M+2+r\leq N+1$ and by Lemma \ref{lemma-2025-7-22}, one has
\begin{align*}
|\textrm{II}_{4,1}|=&\beta\left|\int\Lambda^{M}(h\cdot w)\Lambda^{M+2}adx\right|=\beta\left|\int\Lambda^{M+2+r}(h\cdot w)\Lambda^{M-r}adx\right|\\
\lesssim&\|h\|_{M+2+r}\|a\|_{M-r}\leq C_\epsilon\|h\|^2_{N+1}+\epsilon\|a\|^2_{M-r},\quad\forall\epsilon>0,\\
|\textrm{II}_{4,2}|=&\left|\int\Lambda^M(h\cdot w)\Delta\Lambda^M(h\cdot w)dx\right|
\lesssim\|h\|_{M}\|h\|_{M+2}\lesssim\|h\|^2_{N+1},\\
|\textrm{II}_{4,3}|=&\left|\int\Lambda^M(h\cdot w)\text{div}\Lambda^MR_2dx\right|\lesssim\|h\|_{M}\|R_2\|_{M+1}\\
\lesssim&(1+\|a,u,h\|_{3}^2)\|a,u,h\|_{3}\|a,u,h\|^2_{N},
\end{align*}
and thus
$$
|\textrm{II}_4|\lesssim\epsilon\|a\|_{M-r}^2+C_\epsilon\|h\|_{N+1}^2+(1+\|a,u,h\|_{3}^2)\|a,u,h\|_{3}\|a,u,h\|^2_{N},
$$
for any $\epsilon>0$. For $\textrm{II}_5$ and $\textrm{II}_6$, one deduces by using integrating by parts and $ M+r+2\leq N$ that
\begin{align*}
|\textrm{II}_5|=&\Big|\int(w\cdot\nabla)^2(\Lambda^{M+r}h\cdot w)\Lambda^{M-r}adx\Big|\lesssim\|h\|_{N+1}\|a\|_{M-r}\\
\leq& C_\epsilon\|h\|^2_{N+1}+\epsilon\|a\|^2_{M-r}
\end{align*}
and similarly
$$
|\textrm{II}_6|\lesssim\|h\|_{N+1}\|a\|_{M-r}\leq C_\epsilon\|h\|^2_{N+1}+\epsilon\|a\|^2_{M-r},
$$
for any $\epsilon>0$.
Inserting estimates for $\textrm{II}_1$--$\textrm{II}_7$ into \eqref{pri-8} leads to the conclusion.
\end{proof}



\begin{corollary}\label{corollary2} Suppose \eqref{pri-2} holds and $w$ satisfies the Diophantine condition \eqref{Diophantine0}
with constants $c$ and $r$.
Then, it holds for any $t\in[0,T]$ that
\begin{equation*}\begin{split}
& (2C_1K_1)^{-1}\|u\|_{L-r}^2+C_1^*\beta\|a\|_{M-r}^2+\frac{d}{dt}\int\bigg[2C_{1}^*\Lambda^{M-r-1}\mathrm{div}u\Lambda^{M-r-1}a
\\
&+(w\cdot\nabla)\Lambda^Lh\cdot\Lambda^Lu+\frac{6C_1^*}{K_3}\bigg(\emph{div}_w\Lambda^M(u\times w)\Lambda^Ma+\Lambda^M(h\cdot w)\emph{div}\Lambda^Mu\bigg)\bigg]dx\\
\leq& C\|h\|^2_{N+1}+C(1+\|a,u,h\|_3^2)\|a,u,h\|_3\|a,u,h\|^2_{N},
\end{split}\end{equation*}
where $K_1$ and $K_3$ are as in Lemma \ref{Diophantine-inequality} and Corollary \ref{CorPoincare}, respectively,
$C_1$ is a constant depending only on $N$, and constants $C_1^*$ and $C$ depend only on $N, |w|, \beta, K_3,$ and $Q$.
\end{corollary}

\begin{proof}
Noticing that $M-r\leq N+1$, it follows from Lemma \ref{lemma-2025-6-27} and the Young inequality that
\begin{equation}\label{INEQ1}
\begin{split}
&\|\mathrm{div}\Lambda^{M-r-1}u\|_0^2+\frac{d}{dt}\int\Lambda^{M-r-1}\mathrm{div}u\Lambda^{M-r-1}adx \\
\leq&\beta\|a\|_{M-r}^2
+\|h\|_{M-r}\|a\|_{M-r}+C(1+\|a,u,h\|_3^2)\|a,u,h\|_{3}\|a,u,h\|^2_{N}\\
\leq&2\beta\|a\|_{M-r}^2+C_\beta\|h\|_{N+1}^2+C_\beta(1+\|a,u,h\|_N^2)\|a,u,h\|_{3}\|a,u,h\|^2_{N},
\end{split}
\end{equation}
for a positive constant $C_\beta$ depending only on $N, |w|, \beta,$ and $Q$. Using \eqref{EQUIVN} and \eqref{pri-2}, it follows from
Lemma \ref{Diophantine-inequality} that
\begin{align*}
  \|u\|_{L-r}^2\leq C_1\|\Lambda^{L-r}u\|_0^2\leq C_1 K_1 \|w\cdot\nabla\Lambda^Lu\|_0^2,
\end{align*}
where $C_1$ is a positive constant depending only on $N$ and $K_1$ is the constant as in Lemma \ref{Diophantine-inequality}.
Thanks to this, it follows from Lemma \ref{lemma-pri3} and $L\leq M-r-1$ that
\begin{equation}\label{INEQ2}
\begin{split}
&(2C_1K_1)^{-1}\|u\|_{L-r}^2+\frac{d}{dt}\int(w\cdot\nabla)\Lambda^Lh\cdot \Lambda^Ludx\\
\leq& C_1^*\|\mathrm{div}\Lambda^{M-r-1}u\|^2_0+C_1^*\|h\|^2_{N+1}+C_1^*(1+\|a,u,h\|_3^2)\|a,u,h\|_3\|a,u,h\|^2_{N},
\end{split}
\end{equation}
for a positive constant $C_1^*$ depending only on $N, |w|, \beta, Q$.
Recalling the definition of $a$ in (\ref{a}), it follows from \eqref{junzhi1} that
$\int adx=0$. Thanks to this, it follows from Corollary \ref{CorPoincare} and Lemma \ref{lemma-pri2} that
\begin{equation}\begin{split}\label{INEQ3}
\beta K_3\|a\|_{M-r}^2+\frac{d}{dt}&\int\Big(\mathrm{div}_w\Lambda^M(u\times w)\Lambda^Ma+\Lambda^M(h\cdot w)\mathrm{div}\Lambda^Mu\Big)dx\\
\leq&C(1+\|a,u,h\|_{3}^2)\|a,u,h\|_{3}\|a,u,h\|^2_{N}+ C_\epsilon\|h\|^2_{N+1}\\
&+\epsilon(\|a\|^2_{M-r}+\|\mathrm{div}\Lambda^{M-r-1}u\|^2_0),
\end{split}\end{equation}
for any $\epsilon>0$, where $K_3$ is the constant as in Corollary \ref{CorPoincare}, $C$ and $C_\epsilon$ are positive
constants depending only on $N, |w|, \beta,$ and $Q$ (with $C_\epsilon$ depending also on $\epsilon$).
Multiplying \eqref{INEQ1} and (\ref{INEQ3}), respectively, by $2C_1^*$ and $\frac{6C_1^*}{K_3}$, summing the resultants
up, together with \eqref{INEQ2}, and choosing $\epsilon=\frac{K_3}{6}\min\{\beta,1\}$ lead to the conclusion.
\end{proof}

\begin{lemma}\label{lemma-pri4}
Assume that $w$ satisfies the Diophantine condition \eqref{Diophantine0} with constants $c$ and $r$, and that
\begin{eqnarray}
&&\frac12<\rho<\frac32\mbox{ on }\mathbb T^3\times[0,T],\quad \sup_{0\leq t\leq T}\mathcal E_{N+d}(t)\leq \Gamma,  \label{M}\\
&&\sup_{0\leq t\leq T}\left[\big(1+\mathcal E_N(t)\big)\mathcal E_{L-r}^{\frac{d-(N+r-L)}{2(N+d+r-L)}}(t)\mathcal E_{N+d}^{\frac{N+r-L}{N+d+r-L}}(t)\right]\leq\varepsilon, \label{epsilon1}
\end{eqnarray}
for two positive constants $\Gamma$ and $\varepsilon$, where $\mathcal E_j(t)$ is as in (\ref{Ej}).

Then, there are positive constants $\varepsilon_1^*$ and $C_{*}\geq1$ depending only on $N$, $d$, $L$, $|w|$, $r$, $c$, $\beta,$ $\alpha_1$, $\alpha_2,$ and $Q$, such that
the following estimates hold
\begin{eqnarray*}
&&\mathcal E_N(t)\leq C_*\mathcal E_{N,0}\left[1+\left(\frac{\mathcal E_{N,0}}{\Gamma}\right)^{\frac{N+r-L}{d}} t\right]^{-\frac{d}{N+r-L}}, \\
&&\mathcal E_{N+d}(t)\leq C_*\mathcal E_{N+d,0} \exp\left\{C_{*} \Gamma^{\frac{N+r-L}{d}}\mathcal E_{N,0}^{\frac{d-2(N+r-L)}{2d}} (1+\mathcal E_{N,0})
  \right\},
\end{eqnarray*}
as long as $\varepsilon\leq\varepsilon_1^*$.
\end{lemma}

\begin{proof}
Denote
$$
\widetilde{\mathcal E}_N(t):=\int e(a+1)dx+\sum_{s=1}^N\left\|\sqrt{\frac{p'}{\rho}}\Lambda^sa\right\|_0^2
+\sum_{s=0}^N\|\sqrt\rho\Lambda^su,\Lambda^sh\|^2_0.
$$
Then, by Lemma \ref{lemma-pri1} and \eqref{pri-1}, one obtains by the Poincar\'e and Young inequalities that
\begin{equation}\begin{split}\label{pri-20}
\frac{d}{dt}\widetilde{\mathcal E}_N(t)+2\kappa\|h\|_{N+1}^2
\lesssim&(1+\|a,u,h\|_3^2)\|a,u,h\|_3\|a,u,h\|^2_N\\
\lesssim&(1+\|a,u,h\|_N^2)\|a,u,h\|_{L-r}\|a,u,h\|^2_N,
\end{split}\end{equation}
for an absolute positive constant $\kappa$.
Note that $e(1)=e'(1)=0$ and $e''(\rho)=\frac{2p'(\rho)}{\rho}>0$ and $a\in(-\frac12,\frac12)$.
One obtains by the Taylor expansion that
 \begin{equation*}\begin{split}
\left(\inf_{s\in(\frac12,\frac32)}\frac{p'(s)}{s}\right)a^2\leq e(a+1)\leq \left(\sup_{s\in(\frac12,\frac32)}\frac{p'(s)}{s}\right)a^2.
\end{split}\end{equation*}
Since $\rho\in(\frac12,\frac32)$ and recalling the definitions of $\alpha_1$ and $\alpha_2$, it is clear that
\begin{equation}\begin{split}\label{2025-7-8-3}
\alpha_1\mathcal E_N(t)=\alpha_1\|a,u,h\|^2_N\leq \widetilde{\mathcal E}_N(t)
\leq\alpha_2\|a,u,h\|^2_N=\alpha_2\mathcal E_N(t).
\end{split}\end{equation}
Denote
\begin{equation*}\begin{split}
X(t):=&\widetilde{\mathcal E}_N(t)+\delta_*\int\bigg[2C_{1}^*\Lambda^{M-r-1}\mathrm{div}u\Lambda^{M-r-1}a
+(w\cdot\nabla)\Lambda^Lh\cdot\Lambda^Lu\\
&+\frac{6C_1^*}{K_3}\bigg({\rm{div}}_w\Lambda^M(u\times w)\Lambda^Ma+\Lambda^M(h\cdot w)\rm{div}\Lambda^Mu\bigg)\bigg]dx
\end{split}\end{equation*}
for a positive constant $\delta_*$ to be determined. Due to \eqref{2025-7-8-3}, one can choose $\delta_*$ sufficiently small such that
\begin{equation}
 \frac{\alpha_1}{2}\mathcal E_N(t)= \frac{\alpha_1}{2}\|a,u,h\|_N^2\leq X(t)\leq2\alpha_2\|a,u,h\|_N^2
 = 2\alpha_2\mathcal E_N(t). \label{pri-21}
\end{equation}
Then, choosing $\delta_*$ sufficiently small,
it follows from \eqref{pri-20} and Corollary \ref{corollary2} that
\begin{equation}\begin{split}\label{pri-22}
\frac{d}{dt}X(t)+&\kappa\|h\|_{N+1}^2+\delta_*\beta C_1^*\|a\|_{M-r}^2+\delta_*(2C_1K_1)^{-1}\|u\|_{L-r}^2\\
\lesssim&(1+\|a,u,h\|_N^2)\|a,u,h\|_{L-r}\|a,u,h\|^2_N.
\end{split}\end{equation}
By the interpolation inequality, it holds that
\begin{equation}
\|a, u, h\|_N^2\leq C\|a, u, h\|_{L-r}^{\frac{2d}{N+d+r-L}}\|a, u, h\|_{N+d}^{\frac{2(N+r-L)}{N+d+r-L}}. \label{INTP}
\end{equation}
Plugging this into \eqref{pri-22} and by assumption (\ref{epsilon1}), one deduces
\begin{align*}
  &\frac{d}{dt}X(t)+\kappa\|h\|_{N+1}^2+\delta_*\beta C_1^*\|a\|_{M-r}^2+\delta_*(2C_1K_1)^{-1}\|u\|_{L-r}^2\\
\leq&C(1+\|a,u,h\|_N^2)\|a,u,h\|_{L-r}\|a, u, h\|_{L-r}^{\frac{2d}{N+d+r-L}}\|a, u, h\|_{N+d}^{\frac{2(N+r-L)}{N+d+r-L}}\\
=&C\left((1+\|a, u, h\|_N^2)\|a, u, h\|_{L-r}^{\frac{d-(N+r-L)}{N+d+r-L}}
\|a, u, h\|_{N+d}^{\frac{2(N+r-L)}{N+d+r-L}}\right)\|a, u, h\|_{L-r}^2\\
=&C\big(1+\mathcal E_N(t)\big)\mathcal E_{L-r}^{\frac{d-(N+r-L)}{2(N+d+r-L)}}(t)\mathcal E_{N+d}^{\frac{N+r-L}{N+d+r-L}}(t)\|a, u, h\|_{L-r}^2\\
\leq&C\varepsilon \|a,u,h\|_{L-r}^2\leq C\varepsilon_1^*\|a,u,h\|_{L-r}^2,
\end{align*}
from which, choosing $\varepsilon_1^*$ such that $C\varepsilon_1^*\leq\delta_{\sharp}
:=\frac12\min\{\kappa,\delta_*\beta C_1^*,\delta_*(2C_1K_1)^{-1}\}$, it follows
\begin{equation}
\begin{split}\label{pri-23}
\frac{d}{dt}X(t)+\delta_\sharp\|a,u,h\|_{L-r}^2\leq0.
\end{split}\end{equation}

It follows from (\ref{M}), \eqref{pri-21}, and (\ref{INTP}) that
$$
X(t)\leq 2\alpha_2\|a, u, h\|_N^2\leq 2\alpha_2C\|a,u,h\|_{L-r}^{\frac{2d}{N+d+r-L}}\Gamma^{\frac{N+r-L}{N+d+r-L}}
$$
and thus
$$
C_\sharp\Gamma^{-\frac{N+r-L}{d}}\left(\frac{X}{\alpha_2}\right)^{\frac{N+d+r-L}{d}}(t)\leq\|a, u, h\|_{L-r}^2(t)
$$
for some positive constant $C_\sharp$. Substituting this into (\ref{pri-23}) yields
$$
\frac{d}{dt}X(t)+\delta_\sharp C_\sharp\Gamma^{-\frac{N+r-L}{d}}\left(\frac{X}{\alpha_2}\right)^{\frac{N+d+r-L}{d}}(t)\leq0,
$$
which implies
\begin{equation*}
  X(t)\leq X_0\left(1+\frac{N+r-L}{d}\frac{\delta_\sharp C_\sharp}{\alpha_2}
  \left(\frac{X_0}{\alpha_2\Gamma}\right)^{\frac{N+r-L}{d}} t\right)^{-\frac{d}{N+r-L}},
\end{equation*}
with $X_0:=X(0)$. Thanks to this, it follows from (\ref{pri-21}) that
\begin{align*}
  \mathcal E_N(t)\leq& \frac{2}{\alpha_1}X(t)\leq
\frac{2}{\alpha_1}X_0\left(1+\frac{N+r-L}{d}\frac{\delta_\sharp C_\sharp}{\alpha_2}
  \left(\frac{X_0}{\alpha_2\Gamma}\right)^{\frac{N+r-L}{d}} t\right)^{-\frac{d}{N+r-L}}  \\
  \leq& \frac{4\alpha_2}{\alpha_1}\mathcal E_{N,0}\left(1+\frac{N+r-L}{d}\frac{\delta_\sharp C_\sharp}{\alpha_2}
  \left(\frac{\alpha_1\mathcal E_{N,0}}{2\alpha_2\Gamma}\right)^{\frac{N+r-L}{d}} t\right)^{-\frac{d}{N+r-L}},
\end{align*}
from which one can easily derive the first conclusion by suitably choosing $C_*$.

We now prove the second conclusion. Denote
$$
Y(t):=\int e(a+1)dx+\sum_{s=1}^{N+d}\left\|\sqrt{\frac{p'}{\rho}}\Lambda^sa\right\|_0^2+\sum_{s=0}^{N+d}\|\sqrt\rho\Lambda^su,\Lambda^sh\|^2_0.
$$
Then, similar to \eqref{2025-7-8-3}, it follows that
\begin{equation}\label{Yt}
  \alpha_1\mathcal E_{N+d}(t)=\alpha_1\|a, u, h\|_{N+d}^2\leq Y(t)\leq \alpha_2\|a, u, h\|_{N+d}^2=\alpha_2\mathcal E_{N+d}(t).
\end{equation}
By Lemma \ref{lemma-pri1} and \eqref{pri-1}, one obtains by the Poincar\'e and Young inequalities that
\begin{equation*}\begin{split}
\frac{d}{dt}Y(t)+2\kappa\|h\|_{N+d+1}^2\lesssim&(1+\|a,u,h\|_3^2)\|a,u,h\|_3\|a,u,h\|^2_{N+d}\\
\lesssim& (1+\|a,u,h\|_N^2)\|a,u,h\|_{N}\|a,u,h\|^2_{N+d},
\end{split}\end{equation*}
for an absolute positive constant $\kappa$. Thanks to this, using the first conclusion and
\eqref{Yt}, one gets
\begin{equation*}
\begin{split}
\frac{d}{dt}Y(t)
\leq& C\left(1+ \mathcal E_{N,0}\right) \mathcal E_{N,0}^\frac12
  \left(1+ \left(\frac{ \mathcal E_{N,0}}{ \Gamma}\right)^{\frac{N+r-L}{d}} t\right)^{-\frac{d}{2(N+r-L)}}Y(t)
\end{split}
\end{equation*}
from which, noticing that $d>2(N+r-L)$, it follows that
\begin{align*}
  Y(t)\leq& Y_0\exp\left\{C \left(1+\mathcal E_{N,0}\right) \mathcal E_{N,0}^\frac12
  \int_0^\infty\left(1+ \left(\frac{ \mathcal E_{N,0}}{ \Gamma}\right)^{\frac{N+r-L}{d}} t\right)^{-\frac{d}{2(N+r-L)}}dt \right\}\\
  =&Y_0\exp\left\{\frac{2C(N+r-L)}{d-2(N+r-L)} \Gamma^{\frac{N+r-L}{d}}\mathcal E_{N,0}^{\frac{d-2(N+r-L)}{2d}} (1+\mathcal E_{N,0})\right\},
\end{align*}
where $Y_0:=Y(0)$. Thanks to this and by (\ref{Yt}), one gets
\begin{align*}
  \mathcal E_{N+d}(t)\leq&\frac{Y(t)}{\alpha_1}\leq\frac{Y_0}{\alpha_1}\exp\left\{C \Gamma^{\frac{N+r-L}{d}}\mathcal E_{N,0}^{\frac{d-2(N+r-L)}{2d}} (1+\mathcal E_{N,0})
  \right\}\\
  \leq&\frac{\alpha_2}{\alpha_1}\mathcal E_{N+d,0} \exp\left\{C \Gamma^{\frac{N+r-L}{d}}\mathcal E_{N,0}^{\frac{d-2(N+r-L)}{2d}} (1+\mathcal E_{N,0})
  \right\},
\end{align*}
from which one can derive easily the second conclusion by suitably choosing $C_*$.
\end{proof}

As a conclusion of this section, we have the following proposition.

\begin{proposition}\label{PROP-KEY}
Assume that $w$ satisfies the Diophantine condition \eqref{Diophantine0} with constants $c$ and $r$. Then, there are positive constants
$\varepsilon_0$ and $C_0$ depending only on $N$, $d$, $L$, $|w|$, $r$, $c$, $\beta,$ $\alpha_1$, $\alpha_2,$ and $Q$, such that
 \begin{eqnarray*}
 && \frac34\leq\rho\leq\frac45\quad\mbox{on }\mathbb T^3\times[0,T], \\
&&\mathcal E_{N}(t)\leq C_0\mathcal E_{N,0}\left[1+
  \left(\frac{\mathcal E_{N,0}}{ \mathcal E_{N+d,0}}\right)^{\frac{N+r-L}{d}} t\right]^{-\frac{d}{N+r-L}}, \\
&&\mathcal E_{N+d}(t)\leq C_0\mathcal E_{N+d,0},
\end{eqnarray*}
for any $t\in[0,T]$, as long as
\begin{equation}
  \mathcal E_{N,0}\left(1+\mathcal E_{N+d,0}^{\frac{2(N+r-L)}{d-2(N+r-L)}}\right)\leq\varepsilon_0.
  \label{ASSUMPTION}
\end{equation}
\end{proposition}

\begin{proof}
By assumption, it is clear that $\mathcal E_{N,0}\leq\varepsilon_0$ and thus it follows from the embedding inequality that
$$
\|\rho_0-1\|_{L^\infty}^2\leq C\|\rho_0-1\|_N^2
\leq C\mathcal E_{N,0}\leq C\varepsilon_0\leq\frac{1}{16},
$$
by choosing $\varepsilon_0$ sufficiently small.
This implies
\begin{equation}
\label{INRHO}
  \frac34\leq\rho_0\leq\frac54\quad\mbox{on }\mathbb T^3.
\end{equation}
Set
\begin{equation}\label{INGamma0}
\Gamma_0:=2eC_*\mathcal E_{N+d,0}>\mathcal E_{N+d,0}.
\end{equation}
Using assumption (\ref{ASSUMPTION}) and since $d>2(N+r-L)$, one deduces
\begin{eqnarray*}
  (1+\mathcal E_{N,0})\mathcal E_{N,0}^{\frac{d-(N+r-L)}{2(N+d+r-L)}}\mathcal E_{N+d,0}^{\frac{N+r-L}{N+d+r-L}}
  =(1+\mathcal E_{N,0})\mathcal E_{N,0}^{\frac{N+r-L}{2(N+d+r-L)}}\left(\mathcal E_{N,0}\mathcal E_{N+d,0}^{\frac{2(N+r-L)}{d-2(N+r-L)}}\right)^{\frac{d-2(N+r-L)}{2(N+d+r-L)}}\\
  \leq(1+\varepsilon_0)\varepsilon_0^{\frac{N+r-L}{2(N+d+r-L)}}\varepsilon_0^{\frac{d-2(N+r-L)}{2(N+d+r-L)}}
  =(1+\varepsilon_0)\varepsilon_0^{\frac{d-(N+r-L)}{2(N+d+r-L)}}\leq\frac{\varepsilon_1^*}{2},
\end{eqnarray*}
by choosing $\varepsilon_0$ sufficiently small. Thus,
\begin{equation}
\label{INPRODUCT}
  (1+\mathcal E_{N,0})\mathcal E_{N,0}^{\frac{d-(N+r-L)}{2(N+d+r-L)}}\mathcal E_{N+d,0}^{\frac{N+r-L}{N+d+r-L}}\leq\frac{\varepsilon_1^*}{2}.
\end{equation}
Denote
\begin{equation*}
\begin{split}
T_*:=
\bigg\{t\in[0,T]\,\bigg|\,&\frac12<\rho<\frac32\mbox{ on }\mathbb T^3\times[0,t], \quad \sup_{0\leq s\leq t}\mathcal E_{N+d}(s)\leq \Gamma_0,\\
 &\text{and } \left.\sup_{0\leq s\leq t}\left[\big(1+\mathcal E_N(s)\big)\mathcal E_{L-r}^{\frac{d-(N+r-L)}{2(N+d+r-L)}}(s)\mathcal E_{N+d}^{\frac{N+r-L}{N+d+r-L}}(s)\right]\leq\varepsilon_1^*\right\}.
\end{split}
\end{equation*}
Due to \eqref{INRHO}, \eqref{INGamma0}, and \eqref{INPRODUCT}, it is clear that $T_*\in(0,T]$.

By the definition of $T_*$, one can apply Lemma \ref{lemma-pri4} to get
\begin{eqnarray}
&&\mathcal E_N(t)\leq C_*\mathcal E_{N,0}\left[1+
  \left(\frac{\mathcal E_{N,0}}{\Gamma_0}\right)^{\frac{N+r-L}{d}} t\right]^{-\frac{d}{N+r-L}},\label{ESTEN0} \\
&&\mathcal E_{N+d}(t)\leq C_*\mathcal E_{N+d,0} \exp\left\{C_{*}\Gamma_0^{\frac{N+r-L}{d}}\mathcal E_{N,0}^{\frac{d-2(N+r-L)}{2d}} (1+\mathcal E_{N,0})\right\},\label{ESTN+D0}
\end{eqnarray}
for any $t\in[0,T_*]$. It is clear from (\ref{ESTEN0}) that
\begin{equation}
\label{ESTEN1}
  \sup_{0\leq t\leq T_*}\mathcal E_N(t)\leq C_*\mathcal E_{N,0}.
\end{equation}
Thanks to this, it follows from the embedding inequality and (\ref{ASSUMPTION}) that
\begin{equation*}
\|\rho-1\|_{L^\infty}^2(t)\leq C\|\rho-1\|_N^2(t)
\leq C\mathcal E_{N}(t)\leq CC_*\mathcal E_{N,0}\leq CC_*\varepsilon_0\leq\frac{1}{16},
\end{equation*}
by choosing $\varepsilon_0$ sufficiently small, for any $t\in[0,T_*]$. Hence,
\begin{equation}
  \label{VERRHO}
  \frac34\leq\rho\leq\frac54\quad\mbox{ on }\mathbb T^3\times[0,T_*].
\end{equation}
Recalling (\ref{INGamma0}) and using assumption (\ref{ASSUMPTION}), one deduces
\begin{align*}
  &C_{*}\Gamma_0^{\frac{N+r-L}{d}}\mathcal E_{N,0}^{\frac{d-2(N+r-L)}{2d}} (1+\mathcal E_{N,0})\\
  =&C_{*}(1+\mathcal E_{N,0})\mathcal E_{N,0}^{\frac{d-2(N+r-L)}{2d}}\left(2eC_*\mathcal E_{N+d,0}\right)^{\frac{N+r-L}{d}}\\
  =&C_{*}\left(2eC_*\right)^{\frac{N+r-L}{d}}(1+\mathcal E_{N,0})
  \left(\mathcal E_{N,0}\mathcal E_{N+d,0}^{\frac{2(N+r-L)}{d-2(N+r-L)}}\right)^{\frac{d-2(N+r-L)}{2d}}\\
  \leq&C_{*}\left(2eC_*\right)^{\frac{N+r-L}{d}}(1+\varepsilon_0)
  \varepsilon_0^{\frac{d-2(N+r-L)}{2d}}\leq1,
\end{align*}
by choosing $\varepsilon_0$ sufficiently small. Thanks to this, it follows from (\ref{ESTN+D0}) that
\begin{equation}
  \sup_{0\leq t\leq T_*}\mathcal E_{N+d}(t)\leq eC_*\mathcal E_{N+d,0}=\frac{\Gamma_0}{2}.
  \label{VEREN+D}
\end{equation}
Combing (\ref{ESTEN1}) with (\ref{VEREN+D}), together with (\ref{ASSUMPTION}), one deduces for any $t\in[0,T_*]$ that
\begin{align*}
  &\big(1+\mathcal E_N(t)\big)\mathcal E_{L-r}^{\frac{d-(N+r-L)}{2(N+d+r-L)}}(t)\mathcal E_{N+d}^{\frac{N+r-L}{N+d+r-L}}(t)\\
  \leq&\big(1+\mathcal E_N(t)\big)\mathcal E_{N}^{\frac{d-(N+r-L)}{2(N+d+r-L)}}(t)\mathcal E_{N+d}^{\frac{N+r-L}{N+d+r-L}}(t)\\
  \leq&\left(1+C_*\mathcal E_{N,0}\right)
  \left(C_*\mathcal E_{N,0}\right)^{\frac{d-(N+r-L)}{2(N+d+r-L)}}\left(eC_*\mathcal E_{N+d,0}\right)^{\frac{N+r-L}{N+d+r-L}}\\
  \leq&\left(eC_*\right)^{\frac32}(1+\mathcal E_{N,0})
  \mathcal E_{N,0}^{\frac{N+r-L}{2(N+d+r-L)}}
  \left(\mathcal E_{N,0}\mathcal E_{N+d,0}^{\frac{2(N+r-L)}{d-2(N+r-L)}}\right)^{\frac{d-2(N+r-L)}{2(N+d+r-L)}}\\
  \leq&\left(eC_*\right)^{\frac32}(1+\varepsilon_0)
  \varepsilon_0^{\frac{d-(N+r-L)}{2(N+d+r-L)}}\leq\frac{\varepsilon_1^*}{2},
\end{align*}
by choosing $\varepsilon_0$ sufficiently small. Thus,
\begin{equation}
  \label{VERPRODUCT}
  \sup_{0\leq t\leq T_*}\left[\big(1+\mathcal E_N(t)\big)\mathcal E_{L-r}^{\frac{d-(N+r-L)}{2(N+d+r-L)}}(t)\mathcal E_{N+d}^{\frac{N+r-L}{N+d+r-L}}(t)\right]\leq\frac{\varepsilon_1^*}{2}.
\end{equation}

Thanks to (\ref{VERRHO}), (\ref{VEREN+D}), and (\ref{VERPRODUCT}), it must have $T_*=T$. Otherwise, if $T_*<T$, then
one can extend to another times $T_{**}\in(T_*,T]$ such that
\begin{equation*}
  \left\{
  \begin{array}{l}
    \frac12<\rho<\frac32\quad\mbox{on }\mathbb T^3\times[0,T_{**}],\quad\displaystyle
    \sup_{0\leq t\leq T_{**}}\mathcal E_{N+d}(t)\leq \Gamma_0,\\
    \displaystyle \sup_{0\leq t\leq T_{**}}\left[\big(1+\mathcal E_N(t)\big)\mathcal E_{L-r}^{\frac{d-(N+r-L)}{2(N+d+r-L)}}(t)\mathcal E_{N+d}^{\frac{N+r-L}{N+d+r-L}}(t)\right]\leq\varepsilon_1^*.
  \end{array}
  \right.
\end{equation*}
By the definition of $T_*$, it then has $T_{**}\leq T_*$, contradicting to $T_{**}>T_*$. This contradiction verifies that $T_*=T$.
Then, the conclusion follows from (\ref{ESTEN0}), (\ref{VERRHO}), (\ref{VEREN+D}), and the definition of $\Gamma_0$, by suitably choosing $C_0$.
\end{proof}

\section{PROOF OF THE MAIN THEOREM}\label{sec4}

This section is devoted to proving Theorem \ref{th:main1}.

\begin{proof}[Proof of Theorem \ref{th:main1}] Based on the solvability of the compressible Euler equations and parabolic system,
for any $(\rho_0,u_0,H_0)\in H^{N+d}(\mathbb T^3)$ with $\inf_{x\in\mathbb T^3}\rho_0>0$ and $\mathrm{div}h_0=0$,
there is a unique local solution $(\rho,u,h)\in C([0,T_0];H^{N+d}(\mathbb T^3))$ with
$\inf_{(x,t)\in\mathbb T^3\times[0,T^\ast]}\rho>0$ to system \eqref{MHD1} for some positive time $T_0$.
By iteratively applying the local well-posedness result, one can
extend the local solution to the maximal time of existence $T_\text{max}$.
Then, by Proposition \ref{PROP-KEY}, it holds that
\begin{eqnarray}
 && \frac34\leq\rho\leq\frac45\quad\mbox{on }\mathbb T^3\times[0,T_\text{max}),\nonumber \\
&&\mathcal E_{N}(t)\leq C_0\mathcal E_{N,0}\left[1+
  \left(\frac{\mathcal E_{N,0}}{ \mathcal E_{N+d,0}}\right)^{\frac{N+r-L}{d}} t\right]^{-\frac{d}{N+r-L}},\nonumber \\
&&\mathcal E_{N+d}(t)\leq C_0\mathcal E_{N+d,0},\label{EN+DAPRI}
\end{eqnarray}
for any $t\in[0,T_\text{max})$.
We claim that $T_\text{max}=\infty$ and the conclusion follows. Otherwise, if $T_\text{max}<\infty$, then it must have
\begin{equation*}
  \varlimsup_{t\rightarrow T_\text{max}}\|a, u, h\|_{N+d}^2(t)=\infty,
\end{equation*}
which contradicts to (\ref{EN+DAPRI}). This completes the proof of Theorem \ref{th:main1}.
\end{proof}

\smallskip
{\bf Acknowledgment.}
This work was supported in part by the the National Natural Science Foundation of China (Grant No. 12371204) and the Key Project of National Natural Science Foundation of China (Grant No. 12131010).

{\bf Conflict of interest statement}: On behalf of all authors, the corresponding author states that there is no conflict of interest.

{\bf Data availability statement}: This paper has no associated data.
\bigskip

\begin{center}

\end{center}

\begin{thebibliography}{99}
\addcontentsline{toc}{section}{References}
\bibitem{Abidi-Paicu} Abidi, H.; Paicu, M.: {\it Global existence for the magnetohydrodynamic system in critical spaces,} Proc. Roy. Soc. Edinburgh Sect. A., \textbf{138} (2008), no. 3, 447--476.

\bibitem{BSV1}
Buckmaster, T.; Shkoller  S.; Vicol, V.: {\it Formation of shocks for 2D isentropic compressible Euler}, Commun.
Pure Appl. Math., {\bf 75} (2022), 2069--2120.
\bibitem{BSV2}
Buckmaster, T.; Shkoller  S.; Vicol, V.: {\it Shock formation and vorticity creation for 3d Euler}, Commun. Pure
Appl. Math., {\bf 76} (2023), 1965--2072.
\bibitem{BSV3}
Buckmaster, T.; Shkoller  S.; Vicol, V.: {\it Formation and development of singularities for the compressible Euler
equations}, EMS Press. DOI 10.4171/ICM2022/210. Proceedings of the International Congress of Mathematicians 2022.

\bibitem{Cabannes-H} Cabannes, H.: {\it Theoretical Magneto-Fluid Dynamics,} Academic, New York (1970)
\bibitem{Cao-Serrano-Shi-Staffilani} Cao-Labora, G.; G\'{o}mez-Serrano, J.; Shi, J.; Staffilani, G.: {\it Non-radial implosion for compressible Euler and Navier-Stokes in $\mathbb{T}^3$ and $\mathbb{R}^3$,} arXiv:2310.05325.


\bibitem{Cao-Wu} Cao, C.; Wu, J.: {\it Global regularity for the 2D MHD equations with mixed partial dissipation and magnetic diffusion,} Advances in Mathematics, \textbf{226} (2011), no. 2, 1803--1822.

\bibitem{Cao-Wu-Yuan}Cao, C.; Wu, J.; Yuan, B.: {\it The 2D incompressible magnetohydrodynamics equations with only
 magnetic diffusion,} SIAM J. Math. Anal., \textbf{46} (2014), no. 1, 588--602.


\bibitem{CGZ} Chen, F.; Guo, B; Zhai, X.: {\it Global solution to the 3-D inhomogeneous incompressible MHD system with discontinuous density,}
Kinetic and Related Models., \textbf{12} (2019), 37--58.

\bibitem{Chen-Zhang-Zhou} Chen, W.; Zhang, Z.; Zhou, J.: {\it Global well-posedness for the 3-D MHD equations with partial diffusion in periodic
domain,} Sci China Math., \textbf{65} (2022), 309--318.

\bibitem{CHRIST1}
Christodoulou, D.: {\it The formation of shocks in 3-dimensional fluids}, EMS Monographs in Mathematics, European
Mathematical Society (EMS), Zurich, 2007.

\bibitem{CHRIST2}
Christodoulou, D.: {\it The shock development problem}, EMS Monographs in Mathematics, European Mathematical
Society (EMS), Zurich, 2019.

\bibitem{Danchin-Mucha} Danchin, R.; Mucha, P.: {\it The Incompressible Navier-Stokes Equations in Vacuum,} Communications on Pure and Applied Mathematics., \textbf{72} (2019), 1351--1385.
\bibitem{Dong-Wu-Zhai} Dong, B.; Wu, J.; Zhai, X.: {\it Global small solutions to a special 2$\frac12$-D compressible viscous non-resistive MHD system,} J. Nonlinear Sci., \textbf{33} (2023), no. 1, Paper No. 21, 37 pp.
\bibitem{Gao-Wu-Xu} Gao, N.; Wu, J.; Xu, F.: {\it Global large solutions to a multi-dimensional compressible magnetohydrodynamic flows with a nonlinear initial constraint,} Calc. Var. Partial Differential Equations., \textbf{64} (2025), no. 3, Paper No. 73.
\bibitem{Hassainia} Hassainia, Z.: {\it On the Global Well-Posedness of the 3D Axisymmetric Resistive MHD Equations,} Ann. Henri Poincar\'{e}., \textbf{23} (2022), 2877--2917.
 \bibitem{Huang-Xin-Yan} Huang, X.; Xin, Z.; Yan, W.: {\it Finite time blowup of strong solutions to the two dimensional MHD equations,} Math. Ann., {\bf 392} (2025), no. 2, 2365--2394.
\bibitem{Jiang-Jang1} Jiang, F.; Jiang, S.: {\it On magnetic inhibition theory in 3D non-resistive magnetohydrodynamic fluids: global existence of large solutions,} Arch. Ration. Mech. Anal., {\bf 247} (2023),  Paper No. 96, 35 pp.
\bibitem{Jiu-Liu-Xie} Jiu, Q; Liu, J; Xie,Y.: {\it Asymptotic stability for $n$-dimensional isentropic compressible MHD equations without magnetic diffusion,} arXiv:2402.09661.
\bibitem{Jiu-Niu-Wu-Yu} Jiu, Q.; Niu, D.; Wu, J.; Xu, X.; Yu, H.: {\it The 2D magnetohydrodynamic equations with magnetic diffusion,} Nonlinearity, \textbf{28} (2015), 3935--3955.
\bibitem{Jiu-Zhao} Jiu, Q.; Zhao, J.: {\it Global regularity of 2D generalized MHD equations with magnetic diffusion,}
 Z. Angew. Math. Phys., \bf{66} \rm(2015), 677--687.
\bibitem{KawashimaS} Kawashima, S.: {\it Systems of a hyperbolic-parabolic composite type, with applications to the equations of magnetohydrodynamics,} Doctoral Thesis: Kyoto University (1984).
\bibitem{Landau-Lifshitz} Landau, L, D.; Lifshitz, E, M.: {\it Electrodynamics of Continuous Media. 2Nd ed,} Pergamon, New York, 1984.
\bibitem{Lei-Zhou} Lei, Z.; Zhou, Y.: {\it BKM's criterion and global weak solutions for magnetohydrodynamics with zero viscosity,} Discrete Contin. Dyn. Syst., \textbf{25} (2009), no. 2, 575-583.
\bibitem{Li-Xu-Zhai} Li, Y.; Xu, H.; Zhai, X.:
{\it Global smooth solutions to the 3D compressible viscous non-isentropic magnetohydrodynamic flows without magnetic diffusion,} J. Geom. Anal., \textbf{33} (2023), no. 8, Paper No. 246, 32 pp.
\bibitem{Lin-Zhang} Lin, F.; Zhang, P.: {\it Global small solutions to an MHD-type system: the three-dimensional case,} Commun. Pure Appl. Math., \textbf{67} (2014),  531--580.
\bibitem{LUKSPECK}
Luk, J.; Speck, J.: {\it Shock formation in solutions to the 2D compressible Euler equations in the presence of
non-zero vorticity}, Invent. Math., {\bf214} (2018), 1-169.

\bibitem{MRRS1}
Merle, F.; Rapha\"el, P.; Rodnianski, I.; Szeftel, J.:
{\it On the implosion of a compressible fluid I: Smooth self-similar inviscid profiles},
{\bf 196} Ann. of Math. (2), 567--778.

\bibitem{Polovin-Demutskii} Polovin, R.-V.; Demutskii, V.-P.: {\it Fundamentals of Magnetohydrodynamics,} Consultants Bureau, New York (1990).
\bibitem{QiaoY} Qiao, Y.: {\it Global solutions of two-dimensional inviscid and resistive magnetohydrodynamics system near an equilibrium,} Math. Methods Appl. Sci., {\bf 47} (2024), no.~9, 7133--7155.
\bibitem{Ren-Xiang-Zhang} Ren, X.; Xiang, Z.; Zhang, Z.: {\it Global well-posedness of 2-D MHD equations without magnetic diffusion in a strip domain,} Nonlinearity., \textbf{29} (2016), no. 4, 1257--1291
\bibitem{SIDERIS}
    Sideris, T. C.: {\it Formation of singularities in three-dimensional compressible fluids}, Comm. Math. Phys., {\bf101} (1985), 475-485.

\bibitem{SuenA} Suen, A.: {\it Existence and uniqueness of low-energy weak solutions to the compressible 3D magnetohydrodynamics equations,} J. Differential Equations., {\bf 268} (2020), no.~6, 2622--2671.
\bibitem{Tan-Wang} Tan, Z.; Wang, Y.: {\it Global well-posedness of an initial-boundary value problem for viscous non-resistive MHD systems,} SIAM J. Math. Anal., {\bf 50} (2018), 1432--1470.
\bibitem{Wang-Xin1} Wang, Y.; Xin, Z.: {\it Global Well-Posedness of Free Interface Problems for the Incompressible Inviscid Resistive MHD,} Commun. Math. Phys., \textbf{388} (2021), 1323--1401.
\bibitem{Wang-Xin} Wang, Y.; Xin, Z.: {\it Global well-posedness of the inviscid heat-conductive resistive compressible MHD in a strip domain,} Commun. Math. Res., 38 (2022), no. 1, 1--27.
\bibitem{Wei-Zhang} Wei, D.; Zhang, Z.: {\it Global well-posedness for the 2-D MHD equations with magnetic diffusion,} Commun. Math. Res., {\bf 36} (2020), no.~4, 377--389.
\bibitem{Wu-Xu-Zhai} Wu, J; Xu, F; Zhai, X: {\it Magnetic Stabilization of Compressible Flows: Global Existence in 3D Inviscid Non-Isentropic MHD Equations,} arXiv:2507.00888.
\bibitem{Wu-Wu} Wu, J.; Wu, Y.: {\it Global small solutions to the compressible 2D magnetohydrodynamic system without magnetic diffusion,} Adv. Math., \textbf{310} (2017), 759--888.
\bibitem{Wu-Zhai} Wu, J.: Zhai, X.:{\it Global small solutions to the 3D compressible viscous non-resistive MHD system,} Math. Models Methods Appl. Sci., {\bf 33} (2023), no.~13, 2629--2656.
\bibitem{Wu-Zhu} Wu, J.; Zhu, Y.: {\it Global well-posedness for 2D non-resistive compressible MHD system in periodic domain,} J. Funct. Anal., \textbf{283} (2022), Paper No. 109602

\bibitem{Xie-Jiu-Liu} Xie, Y.; Jiu, Q.; Liu, J.: {\it Sharp decay estimates and asymptotic stability for incompressible MHD equations without viscosity or magnetic diffusion,} Calc. Var. Partial Differential Equations., {\bf 63} (2024),  Paper No. 191, 24pp.
\bibitem{Xu-Qiao-Fu} Xu, F.; Qiao, L.; Fu, P.: {\it The global solvability of 3-D inhomogeneous viscous incompressible magnetohydrodynamic equations with bounded density,} J. Math.
 Fluid Mech., \textbf{24} (2022), 4--34.
\bibitem{Ye-Yin} Ye, W.; Yin, Z.: {\it Global well-posedness for the non-viscous MHD equations with magnetic diffusion in critical Besov spaces,} Acta Math. Sin. {\bf 38} (2022), no.~9, 1493--1511.
\bibitem{ZhaiX} Zhai, X.: {\it Stability for the 2D incompressible MHD equations with only magnetic diffusion,} J. Differential Equations. {\bf 374} (2023), 267--278.
\bibitem{ZhangJ}Zhang, J.: {\it Local well-posedness of the free-boundary problem in compressible resistive magnetohydrodynamics,} Calc. Var. Partial Differential Equations., {\bf 62} (2023), no.~4, Paper No. 124, 60 pp.
\bibitem{ZhaoY} Zhao, Y.: {\it Global well-posedness for the compressible non-resistive MHD equations in a 3D infinite slab,} Nonlinear Anal., {\bf 227} (2023), Paper No. 113162, 21 pp.
\bibitem{Zhou-Zhu} Zhou, Y.; Zhu, Y.: {\it Global classical solutions of 2D MHD system with only magnetic diffusion on periodic domain,} J. Math. Phys., {\bf 59} (2018), no.~8, 081505, 12 pp.

\bibitem{YIN}
Yin, H.: {\it Formation and construction of a shock wave for 3-D compressible Euler equations with the spherical
initial data}, Nagoya Math. J., {\bf175} (2004), 125--164.
\end{thebibliography}
\end{document}